\documentclass[11pt,a4paper]{article}

\usepackage[margin=1in]{geometry}
\usepackage[T1]{fontenc}
\usepackage{lmodern}
\usepackage[authoryear,round]{natbib}

\usepackage{amsmath,amsfonts,bm}

\def\eqref#1{(\ref{#1})}

\def\1{\bm{1}}

\def\eps{{\epsilon}}

\DeclareMathAlphabet{\mathsfit}{\encodingdefault}{\sfdefault}{m}{sl}
\SetMathAlphabet{\mathsfit}{bold}{\encodingdefault}{\sfdefault}{bx}{n}

\newcommand{\E}{\mathbb{E}}

\newcommand{\R}{\mathbb{R}}

\usepackage{amsmath,amssymb,amsthm,mathtools}
\usepackage{algorithm}
\usepackage{algpseudocode}
\usepackage{booktabs}
\usepackage{array}
\usepackage{enumitem}
\usepackage{microtype}
\usepackage{xcolor}
\usepackage{hyperref}
\usepackage{url}
\hypersetup{hidelinks}

\theoremstyle{plain}
\newtheorem{theorem}{Theorem}[section]
\newtheorem{proposition}[theorem]{Proposition}
\newtheorem{lemma}[theorem]{Lemma}
\newtheorem{corollary}[theorem]{Corollary}
\theoremstyle{definition}
\newtheorem{assumption}[theorem]{Assumption}
\newtheorem{definition}[theorem]{Definition}
\theoremstyle{remark}

\newcommand{\norm}[1]{\left\lVert #1\right\rVert}
\newcommand{\ip}[2]{\left\langle #1,#2\right\rangle}
\newcommand{\argminop}{\operatorname*{arg\,min}}
\newcommand{\Proj}{\operatorname{Proj}}
\newcommand{\op}{\mathrm{op}}
\newcommand{\cO}{\mathcal{O}}

\newcommand{\FD}{\mathrm{FD}}
\newcommand{\PP}{\mathbb P}
\newcommand{\ind}{\mathbf 1}
\newcommand{\prog}{\operatorname{prog}}
\newcommand{\Ortho}{\operatorname{Ortho}}

\title{Optimal Stochastic Bilevel Optimization with First-Order Oracles}

\author{Linxuan Pan \qquad Junchi Yang\\
{\normalsize The Chinese University of Hong Kong, Shenzhen}\\
{\small\texttt{linxuanpan@link.cuhk.edu.cn}\quad
\texttt{yangjunchi@cuhk.edu.cn}}}
\date{}
\hypersetup{pdfauthor={Pan Linxuan and Junchi YANG}}

\begin{document}
\maketitle

\begin{abstract}

We study nonconvex--strongly-convex bilevel optimization under a stochastic
first-order oracle. We introduce MRT-FD, a single-loop first-order method
that simultaneously tracks the upper-level variable, the lower-level solution,
and the auxiliary response arising from implicit differentiation of the
hyperobjective. MRT-FD performs one update of each variable per iteration
and approximates the second-order derivative actions using
order-$p$ finite differences. For any fixed finite smoothness order
$p\ge1$ in the lower-level variable, MRT-FD finds an $\varepsilon$-stationary
point using $\mathcal{O}(\varepsilon^{-4-2/p})$ stochastic gradient
queries. We also prove a matching
$\Omega(\varepsilon^{-4-2/p})$ oracle lower bound. The lower-bound
construction starts from a hard nonconvex minimization chain with a stronger stochastic oracle, and lifts it to a bilevel problem through a
sinusoidal coupling with a scalar lower-level variable. Consequently, the
dependence on $\varepsilon$ is optimal for every fixed finite $p$, closing
the upper--lower complexity gap in this stochastic first-order oracle
setting.
\end{abstract}

\section{Introduction}
\label{sec:introduction}

Bilevel optimization models problems in which the performance of one
decision depends on the solution of another optimization problem. We study
the stochastic nonconvex--strongly-convex (NC--SC) setting
\begin{equation}
\label{eq:intro-problem}
\min_{x\in\R^{d_x}}\ \Phi(x):=f(x,y^\star(x)),
\qquad
y^\star(x):=\argminop_{y\in\R^{d_y}}g(x,y), 
\end{equation}
where $g(x,\cdot)$ is strongly convex and the hyperobjective $\Phi$ may be
nonconvex. This formulation captures applications including hyperparameter
optimization~\citep{franceschi2018bilevel}, data
reweighting~\citep{pan2025scalebio}, and
meta-learning~\citep{finn2017model}. 

Under standard smoothness conditions, $\Phi$ is a smooth nonconvex
function. Writing $H(x,y):=\nabla^2_{yy}g(x,y)$ and
$J(x,y):=\nabla^2_{xy}g(x,y)$, implicit differentiation gives
\begin{equation}
\label{eq:intro-hypergradient}
\nabla\Phi(x)=\nabla_x f(x,y^\star(x))
-J(x,y^\star(x))H(x,y^\star(x))^{-1}\nabla_y f(x,y^\star(x)).
\end{equation}
The hypergradient therefore involves second-order derivatives of $g$,
motivating methods based on Hessian- and Jacobian-vector
products~\citep{ghadimi2018approximation,ji2021bilevel}. When such
second-order stochastic derivative information is available, an
$\cO(\eps^{-4})$ oracle complexity is achievable for finding an
$\eps$-stationary point, matching the $\eps$ dependence inherited from
the lower bound for stochastic nonconvex optimization
\citep{arjevani2023lowerbounds}.

For large-scale problems, however, even Hessian-vector products may be
expensive or unavailable, motivating methods that use only first-order
queries. This restriction substantially changes the known complexity.
Under standard smoothness assumptions, the best first-order methods
achieve an $\cO(\eps^{-6})$ stochastic oracle complexity
\citep{chen2025nearoptimal, gu2026sgha}. Recently, \citet{chen2026faster} show
that additional $p$th-order smoothness with respect to the lower variable
$y$ can improve the upper bound to
$\widetilde{\cO}(\eps^{-4-2/p})$.
As $p$ increases, this rate interpolates between the
$\eps^{-6}$ complexity under base smoothness and the
$\eps^{-4}$ benchmark available with second-order information.

This improvement raises a fundamental question. Existing lower bounds do
not characterize how the complexity of a first-order stochastic
bilevel oracle depends on the available smoothness order $p$. It remains unclear whether the exponent
$4+2/p$ reflects an intrinsic information-theoretic barrier. %
\begin{quote}
\textit{What is the optimal complexity of stochastic NC--SC bilevel
optimization under a first-order oracle, and how does higher-order
smoothness in the lower variable change this complexity?}
\end{quote}

We answer the question above by sharpening both the upper and lower
complexity bounds.

\paragraph{A single-loop first-order algorithm.}
We develop Moving-Target Tracking with Finite Differences (MRT-FD), a
single-loop method that requires  order-$p$ smoothness only for
$g$ with respect to the lower variable $y$, for any $p\ge1$. We introduce
the auxiliary vector $z^\star(x)$ as the unique solution of
\[
H(x,y^\star(x))z^\star(x)=\nabla_y f(x,y^\star(x)).
\]
MRT-FD maintains $y_k$ and $z_k$ to track the moving targets
$y^\star(x_k)$ and $z^\star(x_k)$ while simultaneously updating the outer
variable $x_k$. It uses order-$p$ finite differences to reconstruct the
required second-order derivative actions from first-order queries, while
projection keeps $z_k$ bounded and controls the resulting finite-difference
errors. Our analysis couples objective descent with the tracking errors
$\|y_k-y^\star(x_k)\|^2$ and $\|z_k-z^\star(x_k)\|^2$ through a
moving-target Lyapunov function that jointly controls target motion,
finite-difference bias, and stochastic variance. For fixed
$p$, it achieves
$\cO(\eps^{-4-2/p})$ complexity, matching the $\eps$ dependence of
\citet{chen2026faster} while removing the logarithmic factor.

\paragraph{Lower bound for finite smoothness order $p$.}
We establish a matching lower bound under the same assumptions and oracle
model. The construction embeds a hard nonconvex
chain into a bilevel problem with a scalar lower-level variable through a
sinusoidal coupling. Setting a coupling scale to
$\delta=\eps^{1/p}$ preserves the required order-$p$ smoothness while
keeping the resulting hyperobjective close to the hard chain. We then
construct a stochastic oracle in which
informative observations occur only through a rare Bernoulli event with
probability $\varpi=\Theta(\delta^2\eps^2)$. A random-rotation argument
limits progress along the hidden chain for arbitrary adaptive randomized
algorithms. Combining a
chain of length $T=\Theta(\eps^{-2})$ with the reveal probability gives
the lower bound of
$\Omega(T/\varpi)=\Omega(\eps^{-4-2/p})$. %

\subsection{Related Work}
\label{sec:related-work}

\paragraph{Bilevel optimization with second-order information.}
Classical bilevel methods compute the hypergradient using second-order
information. Iterative differentiation (ITD) differentiates through a
finite lower-level trajectory
\citep{domke2012generic,shaban2019truncated,bolte2021nonsmooth},
whereas approximate implicit differentiation (AID) approximates the
inverse-Hessian--gradient product through linear-system or Neumann-type
solvers
\citep{pedregosa2016hyperparameter,ghadimi2018approximation,
grazzi2020iteration,ji2021bilevel}.

In stochastic NC--SC bilevel optimization, methods such as
stocBiO~\citep{ji2021bilevel} and SUSTAIN~\citep{khanduri2021near}
construct hypergradient estimators using Hessian- and Jacobian-vector
products. SOBA/SABA~\citep{dagreou2022framework} and
MA-SOBA~\citep{chen2024optimal} instead track the associated linear-system
solution jointly with the lower and upper-level variables. These methods are
structurally related to MRT-FD through linear-system tracking, but require
second-order derivative information.

\paragraph{First-order bilevel optimization.}
Fully first-order methods often rely on value-function or penalty
reformulations, including BOME~\citep{liu2022bome}, penalty-based bilevel
gradient methods~\citep{shen2023penalty}, and
F$^2$SA~\citep{kwon2023fullyfirstorder}. Subsequent work improves the
first-order oracle complexity through refined penalty and two-timescale
schemes~\citep{kwon2024complexity,chen2025nearoptimal}. More recently,
SGHA~\citep{gu2026sgha} develops a single-loop method based on a
constrained reformulation of lower-level stationarity and finite-difference
approximations, achieving $\mathcal{O}(\eps^{-6})$ complexity, but requires additional assumptions in the boundedness.

FdeHBO~\citep{yang2023achieving} is the closest structural comparison to
MRT-FD: it tracks the auxiliary linear-system solution using projection
and finite-difference Hessian-vector approximations. Its guarantees,
however, rely on stronger stochastic assumptions, including stochastic
smoothness and unbiased stochastic Hessian estimates with bounded
variance. FdeHBO uses variance reduction and reuses the same sample
across multiple query points, whereas MRT-FD uses fresh independent
samples and order-$p$ finite differences. Their analyses track
different auxiliary errors: FdeHBO uses the current lower-level
linear-system solution, while we directly control
$\|z_k-z^\star(x_k)\|^2$.

\citet{chen2026faster} show that $p$th-order smoothness in the lower
variable can improve the complexity from
$\widetilde{\mathcal{O}}(\eps^{-6})$ to
$\widetilde{\mathcal{O}}(\eps^{-4-2/p})$. Their F$^2$SA-$p$ method uses
high-order finite differences within a penalty formulation over perturbed
lower-level problems, and assumes higher-order smoothness of both $f$ and
$g$ in $y$. In contrast, our additional order-$p$ smoothness condition
is imposed only on $g$.

\paragraph{Lower bounds.}
\citet{liang2023lower} established lower complexity bounds for
convex--strongly-convex and strongly-convex--strongly-convex bilevel
problems, while \citet{dagreou2024lower} derived matching lower bounds
for finite-sum bilevel empirical-risk minimization. Beyond the
strongly-convex lower-level setting, \citet{chen2024finding} showed
hardness results for finding stationary points of the hyperobjective
when the lower-level problem is only convex.
For stochastic NC--SC bilevel optimization, \citet{kwon2024complexity} establish lower bounds under a $y^\star$-aware oracle, where unbiased bounded-variance stochastic gradients are guaranteed only locally around $y^\star(x)$, making their setting not directly comparable to ours. More recently, \citet{chen2026condition} prove an $\Omega(\kappa^{4.5}\eps^{-4})$ lower bound for NC--SC bilevel optimization, where $\kappa$ denotes the condition number, while concurrent work by \citet{gu2026lower} establishes an $\Omega(\kappa^8\eps^{-6})$ lower bound. Our result complements these works by characterizing the effect of higher-order smoothness in the lower variable: for a fixed $p$, we prove a $\Omega(\eps^{-4-2/p})$ lower bound. %

\section{Problem Setting and Oracle Model}
\label{sec:problem}

We consider the bilevel optimization problem
\begin{equation}
\label{eq:problem}
\min_{x\in\R^{d_x}}\ \Phi(x):=f(x,y^\star(x)),
\qquad
y^\star(x):=\argminop_{y\in\R^{d_y}}g(x,y),
\end{equation}
where $f$ and $g$ are the upper- and lower-level objectives, respectively.
We use the notation
\[
  H(x,y):=\nabla^2_{yy}g(x,y),
  \qquad
  J(x,y):=\nabla^2_{xy}g(x,y),
  \qquad
  b(x,y):=\nabla_y f(x,y).
\]
Thus $H(x,y)\in\R^{d_y\times d_y}$ and
$J(x,y)\in\R^{d_x\times d_y}$, so both matrices act on vectors in
$\R^{d_y}$.
We define the auxiliary linear-system solutions at a generic point
$(x,y)$ and at the lower-level minimizer, respectively, as
\[
  z(x,y):=H(x,y)^{-1}b(x,y),
  \qquad
  z^\star(x):=z(x,y^\star(x)).
\]
Under Assumption~\ref{ass:base} below, $H(x,y)\succeq\mu I$, and
implicit differentiation yields the hypergradient
\begin{equation}
\label{eq:hypergradient}
\nabla\Phi(x)
=\nabla_x f(x,y^\star(x))-J(x,y^\star(x))z^\star(x).
\end{equation}
Our goal is to find an approximately stationary point of $\Phi$.
For a possibly random output $x_R$, we refer to
$\E\norm{\nabla\Phi(x_R)}^2\le\eps^2$
as an $\eps$-stationarity guarantee.

\subsection{Assumptions}
\label{subsec:regularity}

We study smooth nonconvex--strongly-convex (NC--SC) bilevel
optimization under the following regularity assumptions.

\begin{assumption}[Base NC--SC regularity]
\label{ass:base}
The following conditions hold for constants $\ell_f,\ell_g,\rho_1,\mu>0$:
\begin{enumerate}[leftmargin=1.5em,label=(\roman*)]
  \item $f$ is differentiable with a jointly $\ell_f$-Lipschitz gradient,
  and $g$ is twice differentiable with a jointly $\ell_g$-Lipschitz
  gradient.
  \item $g(x,\cdot)$ is $\mu$-strongly convex for every $x$.
  \item The Hessian block $\nabla^2_{yy}g$ and the mixed derivative
  $\nabla^2_{xy}g$ are each $\rho_1$-Lipschitz in the joint variable
  $(x,y)$ with respect to the operator norm.
  \item $\norm{\nabla_y f(x,y)}\le B_f$ for all $(x,y)$, and $\Phi$ is bounded below.
\end{enumerate}
\end{assumption}

These assumptions are standard in the
literature~\citep{kwon2024complexity, chen2025nearoptimal, chen2026condition}
and are weaker than those used in some prior
work~\citep{kwon2023fullyfirstorder}.
Assumption~\ref{ass:base} implies that, for every $(x,y)$,
\[
  \norm{z(x,y)}\le B_f/\mu,
  \qquad
  \norm{z^\star(x)}\le B_f/\mu.
\]
We define the closed Euclidean ball centered at the origin with radius
$R_z$ by
\[
  Q:=\{z\in\R^{d_y}:\norm z\le R_z\},
  \qquad
  R_z:=4B_f/\mu.
\]
The algorithm will project its auxiliary updates onto $Q$.
This choice of radius ensures that both the current-point linear-system
solution $z(x,y)$ and the target $z^\star(x)$ lie in $Q$.

\begin{assumption}[Lower-variable order-$p$ smoothness]
\label{ass:order-p}
Fix an integer $p\ge1$ and a constant $L_p>0$.
For each $\ell\in\{x,y\}$, the map $\nabla_\ell g(x,\cdot)$ is $p$ times
differentiable for every $x$ and satisfies
\begin{equation}
\label{eq:p-smooth}
\norm{\nabla_y^p\nabla_\ell g(x,y)-\nabla_y^p\nabla_\ell g(x,y')}_{\op}
\le L_p\norm{y-y'}, \text{ for all } x,y,y'.
\end{equation}
\end{assumption}

This assumption imposes order-$p$ smoothness on the gradients of the
lower-level objective with respect to the lower variable.
For $p=1$, it follows from Assumption~\ref{ass:base}(iii) with
$L_1=\rho_1$, recovering the standard setting.
Beyond the base regularity conditions, we impose an explicit Lipschitz
condition only at order $p$, without separately imposing such conditions
at each intermediate derivative order.
The additional higher-order condition concerns only $g$; in contrast,
prior work also imposes higher-order smoothness on $f$ with respect
to $y$~\citep{chen2026faster}.

\subsection{Stochastic gradient oracle}
\label{subsec:oracle}

We now specify the stochastic gradient oracle used for both the upper
and lower complexity bounds.
At a query point $(x,y)$, a single oracle call returns four gradient
estimates:
\[
  \mathcal O(x,y)=\bigl(\hat F_x(x,y),\hat F_y(x,y),
                       \hat G_x(x,y),\hat G_y(x,y)\bigr).
\]
For $\ell\in\{x,y\}$, $\hat F_\ell(x,y)$ estimates
$\nabla_\ell f(x,y)$, and $\hat G_\ell(x,y)$ estimates
$\nabla_\ell g(x,y)$.
These estimates are assumed to be unbiased and to have bounded variance,
as specified below.

\begin{assumption}[Unbiased oracle with bounded variance]
\label{ass:oracle}
The query rules and oracle responses are measurable.
Let $\mathcal H$ denote the history before an oracle call, and let
$(x,y)$ be an $\mathcal H$-measurable query point.
For each $\ell\in\{x,y\}$, the returned estimates satisfy
\begin{gather}
\E[\hat F_\ell(x,y)\mid\mathcal H]=\nabla_\ell f(x,y),\qquad
\E[\hat G_\ell(x,y)\mid\mathcal H]=\nabla_\ell g(x,y),
\label{eq:oracle-mean}\\
\sum_{\ell\in\{x,y\}}
\E\!\left[
\|\hat F_\ell(x,y)-\nabla_\ell f(x,y)\|^2
+\|\hat G_\ell(x,y)-\nabla_\ell g(x,y)\|^2
\,\middle|\,\mathcal H
\right]
\le \sigma^2,
\label{eq:oracle-variance}
\end{gather}
almost surely, for a constant $\sigma\ge0$.
Each call uses fresh randomness.
Calls at query points fixed by a common conditioning history are
conditionally independent given that history.
No independence is assumed among the four estimates
$\hat F_x,\hat F_y,\hat G_x,\hat G_y$ returned by a single call.
\end{assumption}

For the upper-bound analysis, we use the following shorthand for
minibatch estimates obtained by averaging $m$ fresh oracle calls at a
fixed query point $(x,y)$:
\[
  \bar F_\ell(x,y):=\frac1m\sum_{i=1}^m\hat F_{\ell,i}(x,y),
  \qquad
  \bar G_\ell(x,y):=\frac1m\sum_{i=1}^m\hat G_{\ell,i}(x,y),
  \qquad \ell\in\{x,y\}.
\]
\section{A Single-Loop First-Order Algorithm}
\label{sec:algorithm}

In this section, we introduce our single-loop first-order algorithm and
establish its convergence guarantee. We begin with a finite-difference
procedure for estimating the matrix--vector products appearing in
\eqref{eq:hypergradient}, together with the bias and variance of the finite-difference approximation. 

\subsection{Estimating matrix--vector products from first-order queries}
\label{subsec:stencil}

We approximate the indirect term in \eqref{eq:hypergradient} by tracking
$z^\star(x)$ and forming $J(x,y)z$. At an available pair $(x,y)$, the
tracking update for $z$ is driven by the linear-system residual
$H(x,y)z-b(x,y)$. Thus, both $H(x,y)z$ and $J(x,y)z$ must be estimated
using only first-order information.

For $z\neq0$, define $u:=z/\|z\|$, and set $u:=0$ when $z=0$. The two
matrix--vector products admit the directional-derivative representations
\begin{align}
H(x,y)z
 &=\|z\|\lim_{t\to0}
 \frac{\nabla_y g(x,y+tu)-\nabla_y g(x,y-tu)}{2t},
 \label{eq:exact-h-action}\\
J(x,y)z
 &=\|z\|\lim_{t\to0}
 \frac{\nabla_x g(x,y+tu)-\nabla_x g(x,y-tu)}{2t}.
 \label{eq:exact-j-action}
\end{align}
Evaluating these quotients at a finite radius $\tau>0$ yields a
first-order finite-difference approximation. For example,
\(H(x,y)z\ \approx\
\frac{\|z\|}{2\tau}
\bigl[
\nabla_y g(x,y+\tau u)-\nabla_y g(x,y-\tau u)
\bigr],
\)
and replacing $\nabla_y g$ by $\nabla_x g$ gives the corresponding
approximation of $J(x,y)z$.

To exploit the higher-order regularity in
Assumption~\ref{ass:order-p}, we use a higher-order finite-difference
stencil that combines gradients evaluated at several points
$y+j\tau u$ along the same line. Such interpolation-based
finite-difference formulas are classical; see, e.g.,
\citep{atkinson2001finite, khan1999closed, fornberg1988arbitrarily}.
Let
$
\mathcal I_1=\{-1,1\},
$
and, for $p\ge2$, let
$
\mathcal I_p
=
\{-\lfloor p/2\rfloor,\ldots,\lceil p/2\rceil\}.
$
Choose the unique interpolation weights
$\{a_j\}_{j\in\mathcal I_p}$ satisfying
\begin{equation}
\label{eq:stencil-moments}
\sum_{j\in\mathcal I_p}a_j=0,\qquad
\sum_{j\in\mathcal I_p}j a_j=1,\qquad
\sum_{j\in\mathcal I_p}j^r a_j=0
\quad (2\le r\le p).
\end{equation}
These moment conditions preserve the first-order term in the Taylor
expansion while canceling all polynomial terms of orders
$0,2,\ldots,p$.

For $\ell\in\{x,y\}$, define
\begin{equation}
\label{eq:stencil}
\FD_\ell(x,y;z)
:=
\frac{\|z\|}{\tau}
\sum_{j\in\mathcal I_p}
a_j\nabla_\ell g(x,y+j\tau u).
\end{equation}
We suppress the dependence on $p$ and $\tau$ in the notation.
For $p=1$ and $p=2$, the zero coefficients can be omitted, and
\eqref{eq:stencil} reduces to the same two-point central difference.
For every $p$, the stencil contains at most $p+1$ query locations.
A similar high-order finite-difference construction is used by
\citet{chen2026faster}, but for a different purpose. Their differences
are applied to quantities arising from perturbed lower-level problems
within a penalty-based construction, whereas our differences act
directly on the original lower-level gradients along the direction $z$
to estimate the actions $Hz$ and $Jz$.

\paragraph{Stochastic estimates.}
At each node $j\in\mathcal I_p$, we average $m$ fresh observations of
$\hat G_\ell$ into $\bar G_{\ell,j}$, using separate batches across
nodes. We then define the stochastic finite-difference estimator
\begin{equation}
\label{eq:batch-stencil}
\widehat{\FD}_\ell(x,y;z)
:=
\frac{\|z\|}{\tau}
\sum_{j\in\mathcal I_p}
a_j
\bar G_{\ell,j}(x,y+j\tau u).
\end{equation}
Let $\mathcal H$ be a history that fixes $(x,y,z)$ before these batches
are drawn. Then
\[
\E[\widehat{\FD}_\ell(x,y;z)\mid\mathcal H]
=
\FD_\ell(x,y;z).
\]
Define
$
A_p:=\sum_{j\in\mathcal I_p}a_j^2$. 
Under Assumptions~\ref{ass:order-p} and~\ref{ass:oracle}, for each
$\ell\in\{x,y\}$,
\begin{align}
&\|\FD_\ell(x,y;z)-\nabla_{\ell y}^2g(x,y)z\|
 \le C_p L_p\tau^p\|z\|,
 \label{eq:fd-bias}\\
&\E\!\left[
\|\widehat{\FD}_\ell(x,y;z)-\FD_\ell(x,y;z)\|^2
\,\middle|\,\mathcal H
\right]
 \le
 \frac{A_p\sigma^2\|z\|^2}{m\tau^2},
 \label{eq:fd-variance}
\end{align}
where $C_p$ depends only on the stencil nodes and weights.
Explicit expressions for $C_p$ and $A_p$ are given in
Lemma~\ref{lem:app-fd}. Related expressions follow from classical
finite-difference formulas; see, e.g.,
\citet{li2005general, bai2009representation}. For completeness, we
provide a self-contained proof in the appendix.

Equations~\eqref{eq:fd-bias} and~\eqref{eq:fd-variance} reveal the
bias--variance tradeoff induced by finite differencing. A smaller
$\tau$ decreases the deterministic bias at rate $\tau^p$, but amplifies
the stochastic variance at rate $\tau^{-2}$. Increasing the batch size
$m$ controls this variance amplification. The dependence of $C_p$ and
$A_p$ on the stencil therefore enters the constants in the resulting
complexity bounds.

\begin{samepage}

\subsection{Algorithm}
\label{subsec:algorithm}

We present our algorithm, \emph{Moving-Target Tracking with Finite
Differences} (MRT-FD), in Algorithm~\ref{alg:mrt-fd}. The intuition
behind the single-loop design is similar to that of
SOBA~\citep{dagreou2022framework}: rather than solving the lower-level
problem and the auxiliary linear system to high accuracy at every outer
iteration, the algorithm maintains and updates the variables $x$, $y$,
and $z$ simultaneously.
Compared with SOBA, MRT-FD requires only an unbiased, bounded-variance first-order oracle. Finite differences replace Hessian- and Jacobian-vector products without requiring samplewise smoothness or sample reuse across query points. Projection controls the resulting estimation errors.

\begin{algorithm}[H]
\caption{MRT-FD: projected single-loop tracking with order-$p$ differences}
\label{alg:mrt-fd}
\begin{algorithmic}[1]
\Require Fixed $x_0,y_0$, initialize $z_0=0$. Stepsizes
$\eta_x,\eta_y,\eta_z$. Difference radius $\tau$, batch size $m$,
weights $(\mathcal I_p,a)$. Projection radius $R_z$;
$Q=\{z\in\R^{d_y}:\|z\|\le R_z\}$.
\For{$k=0,\ldots,K-1$}
  \State $z_{k+1}\gets\Proj_Q\!\left[z_k-\eta_z
    \bigl(\widehat{\FD}_y(x_k,y_k;z_k)-\bar F_y(x_k,y_k)\bigr)\right]$.
    \label{line:z-update}
  \State $x_{k+1}\gets x_k-\eta_x\bigl[\bar F_x(x_k,y_k)
  - \widehat{\FD}_x(x_k,y_k;z_{k+1})\bigr]$.
  \State $y_{k+1}\gets y_k-\eta_y \bar G_y(x_k,y_k)$.
    \label{line:y-update}
\EndFor
\State Draw $R$ uniformly from $\{0,\ldots,K-1\}$ and return $x_R$.
\end{algorithmic}
\end{algorithm}
\end{samepage}

\paragraph{The $x$ update: estimate the outer gradient.}
The outer update uses a stochastic approximation of
$\nabla\Phi(x_k)$. In \eqref{eq:hypergradient}, replacing the unavailable
$y^\star(x_k)$ and $z^\star(x_k)$ by the current lower iterate $y_k$ and
the newly updated auxiliary iterate $z_{k+1}$ gives
\begin{equation}
\label{eq:outer-gradient-surrogate}
\nabla_x f(x_k,y_k)-J(x_k,y_k)z_{k+1}.
\end{equation}
The two terms in \eqref{eq:outer-gradient-surrogate} are estimated by
$\bar F_{x,k}:=\bar F_x(x_k,y_k)$ and
$\widehat{\FD}_x(x_k,y_k;z_{k+1})$, respectively. Accordingly, define
$\hat h_k:=\bar F_{x,k}-\widehat{\FD}_x(x_k,y_k;z_{k+1})$ to approximate hypergradient, and then the
outer update can be written as $x_{k+1}=x_k-\eta_x\hat h_k$.
The discrepancy between $\hat h_k$ and $\nabla\Phi(x_k)$ comes from four sources: the tracking errors of $y_k$ and $z_{k+1}$, the finite-difference bias, and the stochastic sampling noise. The $y$- and $z$-updates are designed to control the two tracking errors. Since the outer iterate $x_k$ changes over time, the corresponding targets $y^\star(x_k)$ and $z^\star(x_k)$ also evolve; MRT-FD therefore treats them as moving targets.

\paragraph{The $z$ update: track the inverse-Hessian--gradient product.}
At the available pair $(x_k,y_k)$, the vector $z(x_k,y_k)$ minimizes the
strongly convex quadratic
\begin{equation}
\label{eq:auxiliary-quadratic}
z\ \longmapsto\
\tfrac12 z^\top H(x_k,y_k)z-b(x_k,y_k)^\top z.
\end{equation}
The gradient of this quadratic is the linear-system residual
$H(x_k,y_k)z-b(x_k,y_k)$, whose unique zero is $z(x_k,y_k)$.
Moreover, $z(x_k,y_k)$ approaches the target $z^\star(x_k)$ as $y_k$
approaches $y^\star(x_k)$.
With $\bar F_{y,k}:=\bar F_y(x_k,y_k)$,
Line~\ref{line:z-update} estimates the residual at $z_k$ using
$\widehat{\FD}_y(x_k,y_k;z_k)-\bar F_{y,k}$ and performs one projected
stochastic gradient step. Since $\|z(x_k,y_k)\|\le B_f/\mu$, choosing
$R_z\ge B_f/\mu$ ensures that the current linear-system solution lies in
$Q$. The projection also guarantees $\|z_k\|\le R_z$ for every $k$,
which is crucial for uniformly controlling the finite-difference bias
and variance in \eqref{eq:fd-bias}--\eqref{eq:fd-variance}.

\paragraph{The $y$ update: track the lower solution.}
With $\bar G_{y,k}:=\bar G_y(x_k,y_k)$,
Line~\ref{line:y-update} performs one stochastic gradient step on
$g(x_k,\cdot)$ at $y_k$, namely
$y_{k+1}=y_k-\eta_y\bar G_{y,k}$, to track $y^\star(x_k)$.
This tracking controls the evaluation-point error in both the outer
direction and the coefficients $H(x_k,y_k)$ and $b(x_k,y_k)$ of the
auxiliary linear system.

\subsection{Convergence guarantee}
\label{subsec:complexity}

Let $\Phi_{\inf}:=\inf_x\Phi(x)$, and let $\Delta_0$ denote an upper
bound on the initial objective gap and tracking errors:
\[
\Delta_0\ge\Phi(x_0)-\Phi_{\inf}
  +\|y_0-y^\star(x_0)\|^2+(B_f/\mu)^2,
\]
where the last term bounds the squared auxiliary tracking error at
$z_0=0$. Let $\alpha>0$ and $C_\star\ge1$ be the constants specified in
Theorem~\ref{thm:main}. We set the projection radius to
$R_z=4B_f/\mu$, choose the tracking stepsizes
$\eta_y=\eta_z=1/(2\ell_g)$, and set the outer stepsize
$\eta_x=\alpha$.
For a target accuracy $0<\eps\le1$, choose the finite-difference radius,
minibatch size, and number of iterations as
\begin{equation}
\label{eq:main-parameters}
\tau=\left(\frac{\eps^2}{3C_\star}\right)^{1/(2p)},\qquad
m=\left\lceil\frac{3C_\star}{\eps^2\tau^2}\right\rceil,\qquad
K=\left\lceil\frac{3C_\star(1+\Delta_0)}{\eps^2}\right\rceil.
\end{equation}
The role of the smoothness order $p$ is apparent from these choices.
Since the finite-difference bias scales as $\tau^{2p}$, a larger $p$
allows a larger difference radius $\tau$ while maintaining the same
target bias. This reduces the variance amplification $1/\tau^2$ and,
consequently, the required minibatch size $m$. The resulting convergence
and oracle complexity are summarized below.

\begin{theorem}
\label{thm:main}
Suppose Assumptions~\ref{ass:base}--\ref{ass:oracle} hold.
There exist constants $\alpha>0$ and $C_\star\ge1$, depending only on
$p$ and $\ell_f,\ell_g,\rho_1,\mu,B_f,L_p,\sigma$, such that the
following holds.
Run Algorithm~\ref{alg:mrt-fd} from fixed $x_0,y_0$ and $z_0=0$, with
$R_z=4B_f/\mu$, $\eta_y=\eta_z=1/(2\ell_g)$, and $\eta_x=\alpha$.
For any $0<\tau\le1$ and positive integers $m,K$, the output satisfies
\begin{equation}
\label{eq:main-finite-time}
\E\|\nabla\Phi(x_R)\|^2
\le C_\star\left(\frac{\Delta_0}{K}
   +\tau^{2p}+\frac{1}{m\tau^2}\right).
\end{equation}
With the parameter choices in \eqref{eq:main-parameters}, we obtain
$\E\|\nabla\Phi(x_R)\|^2\le\eps^2$ using at most
\[
(2p+5)mK=\cO\!\left(\eps^{-4-2/p}\right)
\]
oracle calls, for any fixed finite $p$.
\end{theorem}

Appendix~\ref{app:upper-proof} provides the proof and explicit admissible
choices of $\alpha$ and $C_\star$. For $p=1$,
Theorem~\ref{thm:main} gives an $\cO(\eps^{-6})$ sample complexity,
matching existing upper bounds~\citep{chen2025nearoptimal, gu2026lower}.
The single-loop method of \citet{gu2026sgha} attains the same dependence
under additional boundedness assumptions.
For $p\ge2$, our rate matches the $\eps^{-4-2/p}$ dependence of
\citet{chen2026faster} while removing the logarithmic factor. Moreover, \citet{chen2026faster} assumes higher-order smoothness of both $f$ and $g$ in $y$, as their finite-difference estimator is applied to the gradient of a penalty function that combines the upper- and lower-level objectives.
\section{Matching Lower Bound under First-Order Oracles}
\label{sec:lower-bound}

We first recall a hard instance construction for stochastic nonconvex minimization and establish its
lower bound under a stronger oracle notion. We then lift this construction
to a bilevel instance and show that the bilevel oracle call can be
simulated by the oracle in the minimization hard instance. %

\subsection{A stochastic nonconvex minimization hard instance}
\label{subsec:lb-first-order}

Consider a smooth objective $P:\R^d\to\R$ whose gradient Lipschitz
constant is bounded by a numerical constant and $P(0)-\inf P\le1$.  An ordinary stochastic-gradient
oracle returns $\widehat{\nabla P}(x)$ such that, for each query $x$
selected from the history $\mathcal H$,
\[
  \E[\widehat{\nabla P}(x)\mid\mathcal H]=\nabla P(x),\qquad
  \E[\|\widehat{\nabla P}(x)-\nabla P(x)\|^2\mid\mathcal H]\le1.
\]
The construction below recovers the standard minimization lower bound under a strengthened oracle and provides the functions and mechanism used later in our bilevel hard instance.

\paragraph{The objective and its decomposition.}
The standard smooth chain
$C_T(v)=\sum_{i=1}^T\psi_i(v)$, $v\in\R^T$, couples neighboring
coordinates: each $\psi_i$ depends only on $(v_{i-1},v_i)$, with $v_0=1$
\citep{carmon2020stationary}.
At a point supported on the first $k$ coordinates, its gradient is
supported on at most the first $k+1$.  Moreover, $|v_T|\le1$ implies
$\|\nabla C_T(v)\|\ge1$.  Thus a small gradient requires reaching the
last coordinate, while gradients expose at most one new coordinate
at a time along such queries.

Following~\citep{arjevani2023lowerbounds}, we introduce smooth weights
$\Theta_i(v)\in[0,1]$ to isolate the component that can reveal the next
coordinate:
\begin{equation}
\label{eq:lb-chain-split}
  W_T(v):=-\sum_{i=1}^T\Theta_i(v)\psi_i(v),\qquad
  B_T(v):=C_T(v)+W_T(v).
\end{equation}
The weight $\Theta_i$ equals one if $|v_j|\le1/4$ for every $j\ge i$,
and zero if some $|v_j|\ge1/2$ with $j\ge i$.
For queries supported on the first $k$ coordinates, $\nabla B_T$ stays
within those coordinates, whereas $\nabla W_T$ can reach coordinate
$k+1$.  The specific weights also keep $W_T$, its gradient, and its
Hessian bounded independently of $T$.
Appendix~\ref{lb:sec:frontier} gives the formulas for $\psi_i,\Theta_i$
and proves these properties.

Choose $T=\Theta(\eps^{-2})$ and numerical scalings
$a=\Theta(\eps^2)$ and $\lambda_{\rm lb}=\Theta(\eps)$.
Let $U\in\R^{d\times T}$ satisfy $U^\top U=I_T$; in the lower-bound
argument, $U$ is drawn uniformly at random to hide the chain directions
from an adaptive algorithm~\citep{arjevani2023lowerbounds, woodworth2017lower}. Define the soft compression map
\[
  \rho_R(w):=\frac{w}{\sqrt{1+\|w\|^2/R^2}},
  \qquad R=230\sqrt T,
\]
and the corresponding latent coordinate
$v(x):=U^\top\rho_R(x/\lambda_{\rm lb})$.
The hard objective and the pair $(A,u)$ are
\begin{equation}
\label{eq:lb-pair-construction}
\begin{aligned}
  P(x)&:=aC_T(v(x))+\frac{a}{10\lambda_{\rm lb}^2}\|x\|^2,\\
  u(x)&:=aW_T(v(x)),\qquad A(x):=P(x)+u(x).
\end{aligned}
\end{equation}
Here $u:\R^d\to\R$ carries the next-coordinate information and
$P=A-u$. The soft compression bounds the points seen by the chain,
while the quadratic term prevents large-norm queries from having a small
gradient. Although $A,u,P$ depend on the chosen $U$ through $v(x)$, we
omit this dependence from the notation. Appendix~\ref{lb:sec:scaling}
verifies the scaling constants.

\paragraph{The gradient oracle and a comparison oracle.}
At each call, draw a fresh $\zeta\sim\mathrm{Bernoulli}(\varpi)$ and define the stochastic gradient estimator
\begin{equation}
\label{eq:lb-first-order-oracle}
  \widehat{\nabla P}(x)
  :=\nabla A(x)-\frac{\zeta}{\varpi}\nabla u(x).
\end{equation}
Since $\E[\zeta/\varpi]=1$, this estimator is unbiased, while the informative correction $\nabla u(x)$ is suppressed with probability $1-\varpi$.
For the lower-bound argument, we introduce a richer oracle response,
\begin{equation}
\label{eq:lb-comparison-oracle}
  \mathcal R_\varpi(x)
  :=\left(\nabla A(x),\frac{\zeta}{\varpi}u(x),
          \frac{\zeta}{\varpi}\nabla u(x)\right),
\end{equation}
which we call the \emph{comparison oracle}. It is stronger than
\eqref{eq:lb-first-order-oracle}, since subtracting its third component
from its first recovers the stochastic gradient estimate in one call.
The additional scalar component $\frac{\zeta}{\varpi}u(x)$ is included
because it will be used to reproduce the lower-gradient oracle
$\hat G_y$ in the bilevel construction. The next proposition establishes
a lower bound directly for this stronger oracle.

\begin{proposition}[Hard objective and comparison oracle]
\label{prop:lb-comparison}
There are numerical constants $c,C,\eps_0>0$ such that, for every
$0<\eps<\eps_0$ and $0<\varpi\le1/2$, the construction
\eqref{eq:lb-pair-construction} admits $T=\Theta(\eps^{-2})$, $d\ge T$,
and a distribution of $U$ with the following properties.
Both $\nabla A$ and $\nabla u$ are $C$-Lipschitz,
$P(0)-\inf P\le1/4$, and
\begin{equation}
\label{eq:lb-frontier-size}
  \sup_x|u(x)|\le C\eps^2,\qquad
  \sup_x\|\nabla u(x)\|\le C\eps.
\end{equation}
Every adaptive randomized algorithm using at most $cT/\varpi$ calls to
$\mathcal R_\varpi$, with arbitrary queries and final output, satisfies
\begin{equation}
\label{eq:lb-common-progress}
  \PP\bigl(\|\nabla P(x_{\rm out})\|\ge2\eps\bigr)\ge\frac34.
\end{equation}
The rotation $U$ is independent of the algorithm seed and oracle draws;
probability includes all three sources of randomness.
\end{proposition}

The two vector responses in $\mathcal R_\varpi$ arise from
$U\nabla B_T(v)$ and $U(\zeta/\varpi)\nabla W_T(v)$. %
Along an exposed coordinate prefix, these vectors can add at most one
direction, and only when $\zeta=1$
(Corollary~\ref{lb:cor:two-channel-zero-chain}), giving a mean waiting
time $1/\varpi$ between opportunities for progress.
For arbitrary queries, the proof conditions on the past latent coordinates
$v(x_t)$, oracle draws, and exposed directions, which together determine
all past scalar responses.
The residual parts of unrevealed directions orthogonal to past queries
and exposed directions remain spherically symmetric, allowing
high-dimensional concentration to control accidental progress
(Lemmas~\ref{lb:lem:conditional-haar}--\ref{lb:lem:finite-random-rotation}).
The same geometric bound covers the final output as an additional point
without an oracle call (Corollary~\ref{lb:cor:terminal-output}).
Appendix~\ref{lb:sec:comparison} gives the full proof.

Choosing $\varpi=2(1+C^2)\eps^2$ recovers the classical
$\Omega(\eps^{-4})$ lower bound for stochastic nonconvex minimization
under an ordinary first-order oracle; see
Appendix~\ref{lb:sec:first-order}.  More importantly, the comparison-oracle
construction enables a stronger bilevel lower bound: after lifting the
hard instance to a bilevel problem, we can choose a even smaller reveal
probability $\varpi$ while preserving bounded variance, which increases
the number of oracle calls required.

\subsection{A bilevel hard instance}
\label{subsec:lb-embeddings}

We use the same pair $(A,u)$ to construct a bilevel hard instance and its
primitive first-order oracle. For a scale $\eps\le\delta\le1$, let
$x\in\R^d$ and let the lower-level variable be scalar, $y\in\R$. Define
the upper- and lower-level objectives by
\begin{equation}
\label{eq:lb-scalar-example}
  f_\delta(x,y):=A(x)+y,\qquad
  g_\delta(x,y):=\tfrac12y^2+\delta\sin(y/\delta)u(x).
\end{equation}
At a query point $(x,y)$, draw one response from the comparison oracle,
\(
(r_1,r_2,r_3)=\mathcal R_\varpi(x)
=
\left(
\nabla A(x),\frac{\zeta}{\varpi}u(x),
\frac{\zeta}{\varpi}\nabla u(x)
\right)\), 
and transform it into the primitive bilevel oracle response:
\begin{equation}
\label{eq:lb-scalar-oracle}
\hat F_x=r_1,\qquad \hat F_y=1,\qquad
\hat G_x=\delta\sin(y/\delta)r_3,\qquad
\hat G_y=y+\cos(y/\delta)r_2.
\end{equation}
These are
exactly the stochastic gradients of $f_\delta$ and $g_\delta$, with the
same Bernoulli variable $\zeta$ appearing in both lower-level gradient
blocks. Hence, one comparison-oracle call exactly simulates one primitive
bilevel oracle call, including its joint randomness.

The terms carrying the hidden correction $u$ are
$\delta\sin(y/\delta)\nabla u(x)$ and $\cos(y/\delta)u(x)$.
Their joint norm is $O(\delta\eps)$ uniformly over $(x,y)$, since
$\|\nabla u\|=O(\eps)$, $|u|=O(\eps^2)$, and $\delta\ge\eps$.
We refer to this quantity as the \emph{signal size}: it measures the
magnitude of the deterministic terms involving $u$ before multiplication
by $\zeta/\varpi$. These terms contain the information needed to recover
the correction $-\nabla u$ in $\nabla P=\nabla A-\nabla u$.%

The hyperobjective is
$\Phi_\delta(x)=A(x)+y_\delta^\star(x)$, where
\[
  y_\delta^\star+u(x)\cos(y_\delta^\star/\delta)=0,
  \qquad |y_\delta^\star|\le|u(x)|.
\]
Since $|u|/\delta=O(\eps)$, the lower solution is close to $-u$.
Implicit differentiation gives the corresponding gradient estimate. The following proposition verifies that the bilevel construction satisfies Assumption~\ref{ass:base} and that its hyperobjective remains close to the hard minimization objective $P$.

\begin{proposition}[Bilevel gradient and primitive variance]
\label{prop:lb-scalar-transfer}
For each pair in Proposition~\ref{prop:lb-comparison}, every
$\eps\le\delta\le1$, and sufficiently small $\eps$, the family
\eqref{eq:lb-scalar-example} satisfies Assumption~\ref{ass:base} with
uniformly bounded constants, has
$\Phi_\delta(0)-\inf\Phi_\delta\le1$ and $|y_\delta^\star(0)|\le1$, and obeys
\begin{equation}
\label{eq:lb-retained-correction}
  \sup_x\|\nabla\Phi_\delta(x)-\nabla P(x)\|
  \le C_{\rm lift}\eps^5/\delta^2\le C_{\rm lift}\eps^3
\end{equation}
for a numerical $C_{\rm lift}>0$.
The oracle \eqref{eq:lb-scalar-oracle} is conditionally unbiased,
has joint variance at most $C_{\rm lift}\delta^2\eps^2/\varpi$,
and is simulated by one call to $\mathcal R_\varpi$.
\end{proposition}

Appendix~\ref{lb:sec:lift} proves the proposition.
Choosing $\varpi_\delta=\Theta(\delta^2\eps^2)$  keeps the primitive variance at most one.
On the event $\|\nabla P\|\ge2\eps$, the gradient estimate gives
$\|\nabla\Phi_\delta\|\ge3\eps/2$ for small enough $\eps$.
The one-call simulation therefore transfers
Proposition~\ref{prop:lb-comparison} to the bilevel objective, giving
\begin{equation}
\label{eq:lb-two-rates}
  \Omega(T/\varpi_\delta)=\Omega(\eps^{-4}\delta^{-2})
\end{equation}
primitive calls.
The intuition is that the sinusoidal lower-level construction attenuates the primitive signal while preserving the relevant hypergradient information, allowing the same variance budget to support a smaller reveal probability.

\subsection{Lower bound and minimax complexity}
\label{subsec:lb-minimax}

Decreasing $\delta$ increases the higher-order smoothness constants of
$g$ with respect to the lower variable. For the construction in
\eqref{eq:lb-scalar-example}, the order-$p$ smoothness constants in
Assumption~\ref{ass:order-p} are bounded by
$C\eps\delta^{-p}$ for $\ell=x$ and
$C\eps^2\delta^{-(p+1)}$ for $\ell=y$, where $C$ is a numerical
constant independent of $p$. Choosing $\delta=\eps^{1/p}$ keeps these
quantities, as well as the base regularity constants, uniformly bounded.
Substituting this choice into \eqref{eq:lb-two-rates} yields the exponent
$4+2/p$.

For $M\ge1$, let $\mathcal P_p(M)$ be the class of problem--oracle pairs
satisfying Assumptions~\ref{ass:base}--\ref{ass:oracle}, in arbitrary
finite dimensions, with
\begin{gather}
\max\{\ell_f,\ell_g,\rho_1,L_p,B_f,\sigma,\mu^{-1}\}\le M,
\label{eq:lb-class-envelope}\\
\Phi(0)-\Phi_{\inf}\le1,\qquad \|y^\star(0)\|\le1.
\label{eq:lb-class-initial}
\end{gather}
Queries and the final output may be arbitrary measurable functions of
past oracle responses and the algorithm's internal randomness.

\begin{theorem}[Lower bound for bilevel problems]
\label{thm:lower}
There exist numerical constants $M_0\ge1$, $c>0$, and $\eps_0>0$ such that
the following holds for every integer $p\ge1$ and $0<\eps<\eps_0$.
For any adaptive randomized algorithm using at most
$c\eps^{-4-2/p}$ primitive oracle calls, there is a pair in
$\mathcal P_p(M_0)$ on which its output satisfies
\begin{equation}
\label{eq:lb-main-moment}
\E\|\nabla\Phi(x_{\rm out})\|^2>\eps^2.
\end{equation}
\end{theorem}

Appendix~\ref{lb:sec:main-proof} gives the proof.
The hard instances have $d_y=1$ and can be constructed with
$d_x=O(\eps^{-6-2/p}\log(1/\eps))$.
Theorem~\ref{thm:lower} shows that the upper bound in
Theorem~\ref{thm:main} is tight in its dependence on $\eps$ for every
fixed finite $p$ and fixed problem constants.

\subsection*{AI use statement}

We used OpenAI GPT-5.5 to assist with literature search and reference
checking, manuscript drafting and revision, and selected mathematical
computations and proof checking. The algorithmic design and the main
convergence proof were developed by the authors; GPT-5.5 was used to help
verify intermediate calculations, and
improve the presentation. The initial lower-bound construction emerged
through iterative discussions between the authors and GPT-5.5. The authors
subsequently simplified the argument, reorganized the proof, and developed
the current comparison-oracle and bilevel-lifting framework. The authors
take full responsibility for the final manuscript, including all
mathematical claims, proofs, references, and interpretations.

\bibliography{references}
\bibliographystyle{plainnat}

\clearpage
\appendix
\section{Proof of Theorem~\ref{thm:main}}
\label{app:upper-proof}

This appendix proves the guarantee for Algorithm~\ref{alg:mrt-fd} with a
common difference radius and batch size.  We first collect constants, then
establish regularity, control the three updates, and sum a Lyapunov
inequality.  Throughout, Assumptions~\ref{ass:base}--\ref{ass:oracle}
hold, and the initial points $x_0,y_0$ are fixed with $z_0=0$.

\subsection{Constants and parameter choices}
\label{app:parameters}

Write $L=\ell_g$ and use $\Phi_{\inf}=\inf_x\Phi(x)$, as in
Section~\ref{subsec:complexity}.
Define the geometry constants
\begin{equation}
\label{eq:app-geometry-constants}
q_0:=\frac{B_f}{\mu},\quad
L_y:=\frac{L}{\mu},\quad
B:=\ell_f+\rho_1q_0,\quad
L_z:=\frac{(1+L_y)B}{\mu},\quad
L_\Phi:=(1+L_y)B+LL_z.
\end{equation}
The projection radius is $R_z=4q_0$, as in Section~\ref{sec:problem}.
The chosen nodes and weights define finite constants
\begin{equation}
\label{eq:app-stencil-constants}
A_p:=\sum_{j\in\mathcal I_p}a_j^2,\qquad
C_p:=\frac{1}{(p+1)!}
       \sum_{j\in\mathcal I_p}|a_j|\,|j|^{p+1}.
\end{equation}

For the lower and auxiliary steps, set
\begin{equation}
\label{eq:app-tracking-constants}
\eta:=\frac{1}{2L},\qquad
r:=\mu\eta,\qquad
\omega:=1+\frac{36B^2}{\mu^2},\qquad
\rho:=\frac r4.
\end{equation}
The remaining constants used in the proof are
\begin{align}
{C_{\rm a}}&:=3\left(1+\frac{3L}{\mu}\right),&
C&:=3L^2+\frac{B^2{C_{\rm a}}}{\omega},&
D&:=\frac{3}{r}\left(L_z^2+\omega L_y^2\right),
\label{eq:app-cdu}\\
N_{\rm b}&:=\frac{9\eta}{\mu},&
N_{\rm v}&:=\frac32(1+\omega)\eta^2.
\label{eq:app-noise-constants}
\end{align}
All step-size denominators below are positive.  In particular, strong
convexity and gradient Lipschitzness imply $0<\mu\le L$, and $0<r\le1/2$.
Choose
\begin{equation}
\label{eq:app-steps}
\eta_y=\eta_z=\eta,\qquad
\eta_x=\alpha:=
\min\left\{\frac{1}{2L_\Phi},
           \sqrt{\frac{\rho}{16CD}}\right\},\qquad
\lambda:=\frac{2\alpha C}{\rho}.
\end{equation}
Define the error and accuracy constants
\begin{align}
S_{\rm b}&:={C_{\rm a}}+\frac{2C N_{\rm b}}{\rho},&
S_{\rm v}&:=2+\frac{2C N_{\rm v}}{\rho},
\label{eq:app-error-constants}\\
C_{\rm b}&:=4S_{\rm b}(C_pL_pR_z)^2,&
C_{\rm v}&:=4S_{\rm v}\sigma^2(1+A_pR_z^2).
\label{eq:app-accuracy-constants}
\end{align}
For the constant in Theorem~\ref{thm:main}, take
\[
C_\star:=\max\left\{1,\frac{4}{\alpha}\max\{1,\lambda\omega\},
                         C_{\rm b},C_{\rm v}\right\}.
\]
These quantities depend only on the problem parameters and $p$, not on
$\eps$, the batch size, or the iteration count.

The errors below use the target at $x_k$, including after the $z$ update:
\begin{align}
\delta_k&:=y_k-y^\star(x_k),&
e_k&:=z_k-z^\star(x_k),&
\bar e_k&:=z_{k+1}-z^\star(x_k), \label{eq:app-errors}\\
T_k&:=\|e_k\|^2+\omega\|\delta_k\|^2,&
\mathcal V_k&:=\Phi(x_k)-\Phi_{\inf}+\lambda T_k.
\label{eq:app-potential}
\end{align}
Since $z_0=0$, we have $\|e_0\|\le q_0$.
Using $\omega\ge1$ and the initial bound in Section~\ref{subsec:complexity},
\begin{equation}
\label{eq:app-initial-bound}
\begin{aligned}
\mathcal V_0
&\le \max\{1,\lambda\omega\}
\left(\Phi(x_0)-\Phi_{\inf}+q_0^2+\|y_0-y^\star(x_0)\|^2\right)\\
&\le \max\{1,\lambda\omega\}\Delta_0.
\end{aligned}
\end{equation}
An upper bound suffices; the algorithm does not
need to query the exact lower solution or the objective value to implement
the parameter rule.

We will prove, for $0<\tau\le1$ and positive integers $m,K$, that
\begin{equation}
\label{eq:app-finite-time}
\frac1K\sum_{k=0}^{K-1}\E\|\nabla\Phi(x_k)\|^2
\le C_\star\left(\frac{\Delta_0}{K}
 +\tau^{2p}+\frac{1}{m\tau^2}\right).
\end{equation}

\subsection{Geometry and errors from the lower iterate}

\begin{lemma}[Regularity of the lower solution and auxiliary target]
\label{lem:app-geometry}
Under Assumption~\ref{ass:base}, the lower solution exists uniquely for every
$x$.  Moreover,
\begin{align}
\|y^\star(x)-y^\star(x')\|
 &\le L_y\|x-x'\|, \label{eq:app-y-lipschitz}\\
\|z(x,y)\|,\ \|z^\star(x)\|&\le q_0, \label{eq:app-z-bounded}\\
\|z^\star(x)-z^\star(x')\|
 &\le L_z\|x-x'\|. \label{eq:app-z-lipschitz}
\end{align}
The hypergradient identity \eqref{eq:hypergradient} holds, and $\nabla\Phi$
is $L_\Phi$-Lipschitz.
\end{lemma}

\begin{proof}
Strong convexity gives a coercive quadratic lower bound for $g(x,\cdot)$,
so continuity gives existence and strong convexity gives uniqueness.
The first-order condition is $\nabla_y g(x,y^\star(x))=0$.
For $a=y^\star(x)$ and $a'=y^\star(x')$, strong monotonicity implies
\[
\mu\|a-a'\|^2
\le
\langle \nabla_y g(x,a)-\nabla_y g(x,a'),a-a'\rangle.
\]
Substitute the two first-order conditions and apply the Lipschitz bound in
$x$.  The right-hand side is at most
$L\|x-x'\|\|a-a'\|$, proving \eqref{eq:app-y-lipschitz}, including the
case $a=a'$.

Differentiability of the Lipschitz gradient and lower strong convexity give
\[
\mu I\preceq H(x,y)\preceq LI,\qquad
\|J(x,y)\|_{\op}\le L,\qquad
\|H(x,y)^{-1}\|_{\op}\le\mu^{-1}.
\]
The bound on $\nabla_y f$ now proves \eqref{eq:app-z-bounded}.
The derivatives of $(x,y)\mapsto\nabla_y g(x,y)$ are $J(x,y)^\top$ and
$H(x,y)$; their continuity follows from the matrix Lipschitz bounds in
Assumption~\ref{ass:base}(iii).
The implicit-function theorem therefore applies locally at each lower
solution.  Uniqueness makes these local solution maps agree.  It follows
that
\[
D_xy^\star(x)=-H(x,y^\star(x))^{-1}J(x,y^\star(x))^\top.
\]
The chain rule proves \eqref{eq:hypergradient}.

For brevity in the remaining estimates, let
$s=(x,y^\star(x))$, $s'=(x',y^\star(x'))$, $q=z^\star(x)$, and
$q'=z^\star(x')$.  Subtracting the two auxiliary systems gives
\[
H(s)(q-q')
=b(s)-b(s')-[H(s)-H(s')]q'.
\]
By Assumption~\ref{ass:base}(iii),
$\|H(s)-H(s')\|_{\op}\le\rho_1\|s-s'\|$.
Use this bound, $\|q'\|\le q_0$, $\|H(s)^{-1}\|_{\op}\le\mu^{-1}$, and
\[
\|s-s'\|\le\|x-x'\|+\|y^\star(x)-y^\star(x')\|
\le(1+L_y)\|x-x'\|.
\]
The resulting bound is
$\|q-q'\|\le B(1+L_y)\|x-x'\|/\mu$, proving
\eqref{eq:app-z-lipschitz}.

Finally, the difference of the hypergradients equals
\[
\nabla_x f(s)-\nabla_x f(s')
-[J(s)-J(s')]q'-J(s)(q-q').
\]
The first two terms have combined norm at most
$B(1+L_y)\|x-x'\|$, and the last term has norm at most
$LL_z\|x-x'\|$.  Their sum gives $L_\Phi$.
\end{proof}

\begin{lemma}[Errors caused by the current lower iterate]
\label{lem:app-mismatch}
For every $x,y,z$, under Assumption~\ref{ass:base},
\begin{align}
\|H(x,y)z^\star(x)-b(x,y)\|
 &\le B\|y-y^\star(x)\|, \label{eq:app-residual-mismatch}\\
\|\nabla_x f(x,y)-J(x,y)z-\nabla\Phi(x)\|
 &\le L\|z-z^\star(x)\|+B\|y-y^\star(x)\|.
\label{eq:app-gradient-mismatch}
\end{align}
\end{lemma}

\begin{proof}
Write $y^\star=y^\star(x)$ and $q=z^\star(x)$.
The identity $H(x,y^\star)q=b(x,y^\star)$ gives
\[
H(x,y)q-b(x,y)
=[H(x,y)-H(x,y^\star)]q+b(x,y^\star)-b(x,y).
\]
The operator-norm Lipschitz bound on $H$ and $\|q\|\le q_0$ give
the bound $(\rho_1q_0+\ell_f)\|y-y^\star\|$, proving
\eqref{eq:app-residual-mismatch}.
For the second bound, subtract \eqref{eq:hypergradient} to obtain
\begin{align*}
&\nabla_x f(x,y)-J(x,y)z-\nabla\Phi(x)\\
&\quad=\nabla_x f(x,y)-\nabla_x f(x,y^\star)
       -[J(x,y)-J(x,y^\star)]q-J(x,y)(z-q).
\end{align*}
The three terms have norms bounded respectively by
$\ell_f\|y-y^\star\|$, $\rho_1q_0\|y-y^\star\|$, and
$L\|z-q\|$.  Their sum is \eqref{eq:app-gradient-mismatch}.
\end{proof}

\subsection{Finite differences and conditional sampling}

\begin{lemma}[Normalized finite-difference error]
\label{lem:app-fd}
The weights in \eqref{eq:stencil-moments} exist uniquely for every integer
$p\ge1$.  Under Assumptions~\ref{ass:order-p} and~\ref{ass:oracle}, the
finite differences and their sampled estimates satisfy
\eqref{eq:fd-bias}--\eqref{eq:fd-variance} with
$C_p,A_p$ from \eqref{eq:app-stencil-constants}.
More precisely, if $x,y,z$ are measurable with respect to a history
$\mathcal F$ and the endpoint batches use fresh calls after that history,
then, for $\ell\in\{x,y\}$,
\begin{align}
\E[\widehat{\FD}_\ell(x,y;z)\mid\mathcal F]
 &=\FD_\ell(x,y;z), \label{eq:app-fd-mean}\\
\|\FD_\ell(x,y;z)-\nabla^2_{\ell y}g(x,y)z\|
 &\le C_pL_p\tau^p\|z\|, \label{eq:app-fd-bias}\\
\E[\|\widehat{\FD}_\ell(x,y;z)-\FD_\ell(x,y;z)\|^2
       \mid\mathcal F]
 &\le \frac{A_p\sigma^2\|z\|^2}{m\tau^2}.
\label{eq:app-fd-var}
\end{align}
These statements include $z=0$.
\end{lemma}

\begin{proof}
The set $\mathcal I_p$ has $p+1$ distinct nodes, including for $p=1$.
Let $l_j(t)$ be its Lagrange basis polynomials.  Differentiating the
interpolation identity for a polynomial of degree at most $p$ shows that
$a_j=l_j'(0)$ satisfies \eqref{eq:stencil-moments}.
The Vandermonde matrix of distinct nodes is invertible, giving uniqueness.

Suppose first that $z\ne0$ and condition on $\mathcal F$.
Write $u=z/\|z\|$ and
$\varphi(t)=\nabla_\ell g(x,y+tu)$.
By Assumption~\ref{ass:order-p}, $\nabla_y^p\nabla_\ell g(x,\cdot)$
is $L_p$-Lipschitz.  Since $\|u\|=1$, the $p$th derivative of
$\varphi$ is also $L_p$-Lipschitz.
Taylor's formula with an integral remainder gives, for every real $t$,
\begin{align*}
\varphi(t)
 &=\sum_{r=0}^{p}\frac{t^r}{r!}\varphi^{(r)}(0)+\mathcal R_p(t),\\
\mathcal R_p(t)
 &=\frac{t^p}{(p-1)!}\int_0^1(1-s)^{p-1}
       [\varphi^{(p)}(st)-\varphi^{(p)}(0)]\,ds.
\end{align*}
In particular, including when $t<0$,
\[
\|\mathcal R_p(t)\|
\le\frac{L_p|t|^{p+1}}{(p-1)!}
     \int_0^1s(1-s)^{p-1}\,ds
=\frac{L_p|t|^{p+1}}{(p+1)!}.
\]
Thus this argument does not require a $(p+1)$st derivative.
Insert the expansion at $t=j\tau$ into \eqref{eq:stencil}.
All polynomial terms except the first derivative vanish by the moment
identities.  The retained term is
$\|z\|\varphi'(0)=\nabla^2_{\ell y}g(x,y)z$.
The remainder has norm at most
\[
\frac{\|z\|}{\tau}
\sum_j|a_j|\frac{L_p|j\tau|^{p+1}}{(p+1)!}
=C_pL_p\tau^p\|z\|.
\]

For each node $j$, set
$\nu_j:=\bar G_{\ell,j}(x,y+j\tau u)-\nabla_\ell g(x,y+j\tau u)$.
Every summand in \eqref{eq:oracle-variance} is nonnegative, so each block
$\hat G_\ell(x,y)$ has conditional variance at most $\sigma^2$.
Averaging $m$ independent calls reduces this bound to $\sigma^2/m$.
Together with \eqref{eq:oracle-mean}, this gives
\[
\E[\nu_j\mid\mathcal F]=0,\qquad
\E[\|\nu_j\|^2\mid\mathcal F]\le\sigma^2/m,\qquad
\E[\langle\nu_i,\nu_j\rangle\mid\mathcal F]=0\quad(i\ne j).
\]
The cross terms vanish because distinct nodes use distinct calls; no
independence between blocks within a single call is required.  Consequently,
\begin{align*}
\E[\|\widehat{\FD}_\ell-\FD_\ell\|^2\mid\mathcal F]
 &=\frac{\|z\|^2}{\tau^2}
       \sum_j a_j^2\E[\|\nu_j\|^2\mid\mathcal F]\\
 &\le\frac{A_p\sigma^2\|z\|^2}{m\tau^2},
\end{align*}
and the same centered-noise representation proves
\eqref{eq:app-fd-mean}.
If $z=0$, both differences and both exact matrix--vector products are
zero by definition, so every assertion follows without endpoint calls.
\end{proof}

We make the sampling convention for Algorithm~\ref{alg:mrt-fd} explicit.
Each occurrence of a center or node average uses a separate batch of $m$
fresh calls, including when query locations coincide.  The $z$ update is
sampled first.  After $z_{k+1}$ is computed, new batches are drawn for the
$x$ update; a final fresh batch is then drawn for the $y$ update.  All
center batches are evaluated at $(x_k,y_k)$, as in the algorithm.
Thus the outer finite difference uses $z_{k+1}$, whereas the auxiliary
finite difference uses $z_k$.  At termination, the output index $R$ is
drawn uniformly and independently of all oracle calls.

Let $\mathcal F_k$ contain all
calls before iteration $k$, let $\mathcal F_k^z$ additionally contain the
calls used to compute $z_{k+1}$, and let $\mathcal F_k^x$ additionally
contain those used to compute $x_{k+1}$.  The final lower-gradient batch
completes $\mathcal F_{k+1}$.  Thus
\[
\mathcal F_k\subseteq\mathcal F_k^z\subseteq\mathcal F_k^x
\subseteq\mathcal F_{k+1}.
\]
In particular, $\PP(R=j\mid\mathcal F_K)=1/K$ for $j=0,\ldots,K-1$.
Write $\E_k[\cdot]=\E[\cdot\mid\mathcal F_k]$ and introduce the uniform
error bounds
\begin{equation}
\label{eq:app-beta-v}
\beta:=C_pL_pR_z\tau^p,\qquad
v:=\frac{\sigma^2}{m}\left(1+\frac{A_pR_z^2}{\tau^2}\right),
\qquad v_y:=\frac{\sigma^2}{m}\le v.
\end{equation}

\begin{lemma}[Conditional decomposition of the three updates]
\label{lem:app-sampling}
Put $H_k=H(x_k,y_k)$, $J_k=J(x_k,y_k)$, and $b_k=b(x_k,y_k)$.
The sampled directions in Algorithm~\ref{alg:mrt-fd} can be written as
\begin{align}
\hat r_k
&:=\widehat{\FD}_y(x_k,y_k;z_k)-\bar F_{y,k}
 =H_kz_k-b_k+d_k^z+\xi_k^z, \label{eq:app-r-decomposition}\\
\hat h_k
&=\nabla_x f(x_k,y_k)-J_kz_{k+1}+d_k^x+\xi_k^x,
\label{eq:app-h-decomposition}\\
\bar G_{y,k}
&=\nabla_y g(x_k,y_k)+\xi_k^y. \label{eq:app-y-decomposition}
\end{align}
Here $d_k^z$ is $\mathcal F_k$-measurable, $d_k^x$ is
$\mathcal F_k^z$-measurable, and $\|d_k^z\|,\|d_k^x\|\le\beta$.
The noises satisfy
\begin{align}
\E[\xi_k^z\mid\mathcal F_k]&=0,&
\E[\|\xi_k^z\|^2\mid\mathcal F_k]&\le v,\\
\E[\xi_k^x\mid\mathcal F_k^z]&=0,&
\E[\|\xi_k^x\|^2\mid\mathcal F_k^z]&\le v,\\
\E[\xi_k^y\mid\mathcal F_k^x]&=0,&
\E[\|\xi_k^y\|^2\mid\mathcal F_k^x]&\le v_y.
\label{eq:app-noise-moments}
\end{align}
\end{lemma}

\begin{proof}
Every $z_k$ belongs to $Q$, since $z_0=0$ and every auxiliary update is
projected.  Define
\begin{align*}
d_k^z&:=\FD_y(x_k,y_k;z_k)-H_kz_k,\\
d_k^x&:=J_kz_{k+1}-\FD_x(x_k,y_k;z_{k+1}).
\end{align*}
Lemma~\ref{lem:app-fd} gives their measurability and bounds.  Notice the
opposite signs: the outer direction subtracts its finite difference.

For $\xi_k^z$, subtract the finite differences and exact center gradients
from their sampled counterparts in $\hat r_k$.  The center term
$\bar F_{y,k}-b_k$ is conditionally centered with variance at most
$\sigma^2/m$ by the $\hat F_y$ block of
\eqref{eq:oracle-mean}--\eqref{eq:oracle-variance} and fresh batching.
It uses a separate batch from the finite difference, so their conditional
variances add, giving at most
\[
\frac{A_p\sigma^2\|z_k\|^2}{m\tau^2}
+\frac{\sigma^2}{m}\le v.
\]
For the outer direction, condition on $\mathcal F_k^z$.  Its input
$z_{k+1}$ is now fixed and lies in $Q$.  Entirely new center and node
batches give the same argument using the $\hat F_x$ block and $z_{k+1}$
in place of $z_k$.
Finally, $\bar G_{y,k}$ is drawn after $\mathcal F_k^x$ at the already
selected point $(x_k,y_k)$.  The $\hat G_y$ block of
\eqref{eq:oracle-mean}--\eqref{eq:oracle-variance} gives its conditional
mean $\nabla_y g(x_k,y_k)$ and variance bound $v_y$ after averaging.
These arguments also cover repeated query locations because their batches
are distinct.
\end{proof}

All expectations used below are finite.  To see this without assuming
bounded iterates, start with the fixed $x_0,y_0$.  The gradients of $f$ and $g$
grow at most linearly by Assumption~\ref{ass:base}(i), while the mean
finite differences have norm at most $LR_z+\beta$ by
Lemma~\ref{lem:app-fd}.  The noises above have uniformly bounded second
moments.  Since $z_k$ is bounded, induction through the three updates
shows that $x_k,y_k$ have finite second moments at every finite $k$.
Lemma~\ref{lem:app-geometry} then gives finite second moments for the
tracking errors and the outer gradient.  Finally, smoothness of $\Phi$
bounds $|\Phi(x)|$ by a constant times $1+\|x\|^2$, so the objective and
potential are integrable as well.

\subsection{One-step descent and tracking estimates}

\begin{lemma}[Projected auxiliary correction at the current outer point]
\label{lem:app-auxiliary}
With $\eta_z=\eta=1/(2L)$, the auxiliary update satisfies
\begin{equation}
\label{eq:app-auxiliary-bound}
\E_k\|\bar e_k\|^2
\le (1-r)\|e_k\|^2
   +\frac{6\eta}{\mu}\bigl(B^2\|\delta_k\|^2+\beta^2\bigr)
   +\eta^2v.
\end{equation}
The target in $\bar e_k$ is $z^\star(x_k)$, before the outer point moves.
\end{lemma}

\begin{proof}
Set $q=z^\star(x_k)$ and
$w_k=H_kq-b_k+d_k^z$.  These are $\mathcal F_k$-measurable, and
Lemmas~\ref{lem:app-mismatch} and~\ref{lem:app-sampling} give
\begin{equation}
\label{eq:app-w-bound}
\|w_k\|^2\le2B^2\|\delta_k\|^2+2\beta^2.
\end{equation}
Because $q\in Q$ and Euclidean projection is nonexpansive,
\[
\|\bar e_k\|^2
\le\|(I-\eta H_k)e_k-\eta w_k-\eta\xi_k^z\|^2.
\]
Take conditional expectation only after this inequality.  The noise is
centered before projection, so
\begin{equation}
\label{eq:app-preprojection}
\E_k\|\bar e_k\|^2
\le\|(I-\eta H_k)e_k-\eta w_k\|^2+\eta^2v.
\end{equation}
No unbiasedness property of the projected iterate is needed.

The eigenvalues of $H_k$ lie in $[\mu,L]$, and $\eta=1/(2L)$ implies
$\|I-\eta H_k\|_{\op}\le1-r$.
Apply Young's inequality
$\|a+b\|^2\le(1+s)\|a\|^2+(1+s^{-1})\|b\|^2$
with $s=r/2$ to the deterministic term in
\eqref{eq:app-preprojection}.  Since $0<r\le1/2$,
\[
(1+r/2)(1-r)^2\le1-r,\qquad 1+2/r\le3/r.
\]
It follows that
\[
\E_k\|\bar e_k\|^2
\le(1-r)\|e_k\|^2+\frac{3\eta^2}{r}\|w_k\|^2+\eta^2v.
\]
Substitute \eqref{eq:app-w-bound} and use $r=\mu\eta$ to obtain
\eqref{eq:app-auxiliary-bound}.
\end{proof}

\begin{lemma}[Outer-direction error and objective descent]
\label{lem:app-outer}
Let $h_k=\nabla\Phi(x_k)$ and
$d_k=\E[\hat h_k\mid\mathcal F_k^z]-h_k$.
With the steps in \eqref{eq:app-steps},
\begin{align}
\E_k\|d_k\|^2
 &\le C T_k+{C_{\rm a}}\beta^2+\tfrac34v, \label{eq:app-outer-bias}\\
\E_k\|\hat h_k\|^2
 &\le2\|h_k\|^2+2C T_k+2{C_{\rm a}}\beta^2+\tfrac52v,
\label{eq:app-outer-second}\\
\E_k\Phi(x_{k+1})
 &\le\Phi(x_k)-\tfrac\alpha2\|h_k\|^2
       +\tfrac{\alpha C}{2}T_k
       +\tfrac{\alpha {C_{\rm a}}}{2}\beta^2+\tfrac{7\alpha}{8}v.
\label{eq:app-outer-descent}
\end{align}
In particular,
$\E_k\|x_{k+1}-x_k\|^2=\alpha^2\E_k\|\hat h_k\|^2$.
\end{lemma}

\begin{proof}
By Lemmas~\ref{lem:app-mismatch} and~\ref{lem:app-sampling},
\[
\|d_k\|\le L\|\bar e_k\|+B\|\delta_k\|+\beta.
\]
Squaring, taking $\E_k$, and applying
Lemma~\ref{lem:app-auxiliary} yields
\begin{align*}
\E_k\|d_k\|^2
&\le3L^2\E_k\|\bar e_k\|^2+3B^2\|\delta_k\|^2+3\beta^2\\
&\le3L^2\|e_k\|^2
 +3B^2\left(1+\frac{6L^2\eta}{\mu}\right)\|\delta_k\|^2
 +3\left(1+\frac{6L^2\eta}{\mu}\right)\beta^2
 +3L^2\eta^2v\\
&=3L^2\|e_k\|^2+B^2{C_{\rm a}}\|\delta_k\|^2+{C_{\rm a}}\beta^2+\tfrac34v.
\end{align*}
The last equality uses $\eta=1/(2L)$ and the definition of ${C_{\rm a}}$.
By definition of $C$ and $T_k$,
$3L^2\|e_k\|^2+B^2{C_{\rm a}}\|\delta_k\|^2\le C T_k$.
This proves \eqref{eq:app-outer-bias}.

Conditional on $\mathcal F_k^z$, the mean of $\hat h_k$ is $h_k+d_k$
and its centered second moment is at most $v$.  Therefore
\[
\E[\|\hat h_k\|^2\mid\mathcal F_k^z]
\le\|h_k+d_k\|^2+v
\le2\|h_k\|^2+2\|d_k\|^2+v.
\]
Taking $\E_k$ and using \eqref{eq:app-outer-bias} proves
\eqref{eq:app-outer-second}.

By $L_\Phi$-smoothness and the outer update,
\begin{align*}
\E[\Phi(x_{k+1})\mid\mathcal F_k^z]
\le\Phi(x_k)-\alpha\langle h_k,h_k+d_k\rangle
 +\frac{L_\Phi\alpha^2}{2}\bigl(\|h_k+d_k\|^2+v\bigr).
\end{align*}
The exact identity
\[
-\langle h,h+d\rangle
=-\tfrac12\|h\|^2-\tfrac12\|h+d\|^2+\tfrac12\|d\|^2
\]
and $\alpha L_\Phi\le1$ give
\[
\E[\Phi(x_{k+1})\mid\mathcal F_k^z]
\le\Phi(x_k)-\frac\alpha2\|h_k\|^2
 +\frac\alpha2\|d_k\|^2+\frac{L_\Phi\alpha^2}{2}v.
\]
Take $\E_k$, insert \eqref{eq:app-outer-bias}, and bound
$L_\Phi\alpha^2/2\le\alpha/2$.
The variance coefficient becomes $3\alpha/8+\alpha/2=7\alpha/8$,
which proves \eqref{eq:app-outer-descent}.
The movement identity follows directly from
$x_{k+1}-x_k=-\alpha\hat h_k$.
\end{proof}

\begin{lemma}[Tracking after the outer point moves]
\label{lem:app-tracking}
With the steps in \eqref{eq:app-steps}, the two tracking errors satisfy
\begin{align}
\E_k\|e_{k+1}\|^2
&\le(1-r/2)\|e_k\|^2
 +\frac{9\eta B^2}{\mu}\|\delta_k\|^2
 +\frac{9\eta}{\mu}\beta^2+\frac32\eta^2v
 +\frac{3L_z^2}{r}\E_k\|x_{k+1}-x_k\|^2,
\label{eq:app-z-tracking}\\
\E_k\|\delta_{k+1}\|^2
&\le(1-r/2)\|\delta_k\|^2+\frac32\eta^2v_y
 +\frac{3L_y^2}{r}\E_k\|x_{k+1}-x_k\|^2.
\label{eq:app-y-tracking}
\end{align}
Consequently, with $T_k$ from \eqref{eq:app-potential},
\begin{equation}
\label{eq:app-joint-tracking}
\E_k T_{k+1}
\le(1-\rho)T_k+N_{\rm b}\beta^2+N_{\rm v}v
   +D\E_k\|x_{k+1}-x_k\|^2.
\end{equation}
\end{lemma}

\begin{proof}
For the auxiliary target, the exact decomposition is
\[
e_{k+1}=\bar e_k-[z^\star(x_{k+1})-z^\star(x_k)].
\]
Apply Young's inequality with $s=r/2$ and use the $L_z$-Lipschitz bound:
\[
\E_k\|e_{k+1}\|^2
\le(1+r/2)\E_k\|\bar e_k\|^2
   +\frac{3L_z^2}{r}\E_k\|x_{k+1}-x_k\|^2.
\]
Insert \eqref{eq:app-auxiliary-bound} and use
$(1+r/2)(1-r)\le1-r/2$ and $1+r/2\le3/2$.
This gives \eqref{eq:app-z-tracking}; no independence between
$\bar e_k$ and the target movement was used.

For the lower variable, write $y_k^\star=y^\star(x_k)$ and define
\[
\widetilde\delta_k
:=\delta_k-\eta\nabla_y g(x_k,y_k).
\]
By the fundamental theorem of calculus and
$\nabla_y g(x_k,y_k^\star)=0$,
\[
\nabla_y g(x_k,y_k)
=\bar H_k\delta_k,\qquad
\bar H_k:=\int_0^1H(x_k,y_k^\star+t\delta_k)\,dt.
\]
The symmetric matrix $\bar H_k$ has eigenvalues in $[\mu,L]$.
Thus $\|\widetilde\delta_k\|\le(1-r)\|\delta_k\|$.
Put $\Delta_k^y=y^\star(x_{k+1})-y^\star(x_k)$.
The lower update gives
$\delta_{k+1}=\widetilde\delta_k-\Delta_k^y-\eta\xi_k^y$.
Both $\widetilde\delta_k$ and $\Delta_k^y$ are
$\mathcal F_k^x$-measurable.  Centering the final fresh noise at this
history therefore yields
\[
\E_k\|\delta_{k+1}\|^2
\le\E_k\|\widetilde\delta_k-\Delta_k^y\|^2+\eta^2v_y.
\]
Young's inequality with $s=r/2$ and the $L_y$-Lipschitz bound imply
\[
\E_k\|\delta_{k+1}\|^2
\le(1+r/2)(1-r)^2\|\delta_k\|^2
  +\frac{3L_y^2}{r}\E_k\|x_{k+1}-x_k\|^2+\eta^2v_y.
\]
Relaxing $(1+r/2)(1-r)^2\le1-r\le1-r/2$ and
$\eta^2v_y\le(3/2)\eta^2v_y$ proves
\eqref{eq:app-y-tracking}.

Finally, multiply \eqref{eq:app-y-tracking} by $\omega$ and add
\eqref{eq:app-z-tracking}.  The only coupling term is absorbed because
\[
\frac{\omega r}{4}
=\frac r4+\frac{9\eta B^2}{\mu}
\ge\frac{9\eta B^2}{\mu}.
\]
The coefficient of $\|\delta_k\|^2$ is therefore at most
$\omega(1-r/4)$, and that of $\|e_k\|^2$ is at most $1-r/4$.
Use $v_y\le v$ and the definitions of $\rho,N_{\rm b},N_{\rm v},D$
to obtain \eqref{eq:app-joint-tracking}.
\end{proof}

\subsection{Lyapunov descent}

\begin{proposition}[A uniform one-step inequality]
\label{prop:app-lyapunov}
Under Assumptions~\ref{ass:base}--\ref{ass:oracle} and the step choices
\eqref{eq:app-steps},
for every $\tau>0$ and positive integer $m$,
\begin{equation}
\label{eq:app-lyapunov-descent}
\E_k\mathcal V_{k+1}
\le\mathcal V_k-\frac\alpha4\|\nabla\Phi(x_k)\|^2
    -\alpha C T_k+\alpha S_{\rm b}\beta^2+\alpha S_{\rm v}v.
\end{equation}
All constants are given in Appendix~\ref{app:parameters}; in particular,
the step sizes do not depend on the target accuracy.
\end{proposition}

\begin{proof}
Add \eqref{eq:app-outer-descent} to $\lambda$ times
\eqref{eq:app-joint-tracking}, and subtract
$\Phi(x_k)+\lambda T_k$.  This gives
\begin{align*}
\E_k\mathcal V_{k+1}-\mathcal V_k
\le{}&-\frac\alpha2\|h_k\|^2
 +\left(\frac{\alpha C}{2}-\lambda\rho\right)T_k
 +\lambda D\E_k\|x_{k+1}-x_k\|^2\\
&+\left(\frac{\alpha {C_{\rm a}}}{2}+\lambda N_{\rm b}\right)\beta^2
 +\left(\frac{7\alpha}{8}+\lambda N_{\rm v}\right)v.
\end{align*}
Use the movement identity and \eqref{eq:app-outer-second} to bound the
third term by
\[
\lambda D\alpha^2
\left(2\|h_k\|^2+2C T_k+2{C_{\rm a}}\beta^2+\tfrac52v\right).
\]
For an explicit check of every coefficient, put
$\chi=CD\alpha^2/\rho$.
Our choice of $\alpha$ implies $0<\chi\le1/16$, and
$\lambda=2\alpha C/\rho$.  Substitution yields
\begin{align*}
\E_k\mathcal V_{k+1}-\mathcal V_k
\le{}&\alpha\left(-\tfrac12+4\chi\right)\|h_k\|^2
 +\alpha C\left(-\tfrac32+4\chi\right)T_k\\
&+\alpha\left[{C_{\rm a}}\left(\tfrac12+4\chi\right)
                  +\frac{2C N_{\rm b}}{\rho}\right]\beta^2\\
&+\alpha\left[\tfrac78+5\chi
                  +\frac{2C N_{\rm v}}{\rho}\right]v.
\end{align*}
The four numerical coefficients obey
\[
-\tfrac12+4\chi\le-\tfrac14,\quad
-\tfrac32+4\chi\le-\tfrac54\le-1,\quad
\tfrac12+4\chi\le\tfrac34\le1,\quad
\tfrac78+5\chi\le\tfrac{19}{16}\le2.
\]
The definitions of $S_{\rm b}$ and $S_{\rm v}$ now give
\eqref{eq:app-lyapunov-descent}.
\end{proof}

\subsection{Finite-time bound and query complexity}

\begin{proof}[Proof of Theorem~\ref{thm:main}]
Take total expectations in \eqref{eq:app-lyapunov-descent}, discard the
nonpositive $-\alpha C T_k$ term, and sum from $k=0$ to $K-1$.
Since $\mathcal V_K\ge0$, this gives
\[
\frac\alpha4\sum_{k=0}^{K-1}\E\|\nabla\Phi(x_k)\|^2
\le\mathcal V_0+K\alpha S_{\rm b}\beta^2+K\alpha S_{\rm v}v.
\]
After division by $\alpha K/4$ and substitution of
\eqref{eq:app-beta-v},
\begin{align*}
\frac1K\sum_{k=0}^{K-1}\E\|\nabla\Phi(x_k)\|^2
&\le\frac{4\mathcal V_0}{\alpha K}
 +4S_{\rm b}(C_pL_pR_z)^2\tau^{2p}
 +\frac{4S_{\rm v}\sigma^2}{m}
     \left(1+\frac{A_pR_z^2}{\tau^2}\right)\\
&\le\frac{4\mathcal V_0}{\alpha K}
 +C_{\rm b}\tau^{2p}+\frac{C_{\rm v}}{m\tau^2}\\
&\le C_\star\left(\frac{\Delta_0}{K}
 +\tau^{2p}+\frac{1}{m\tau^2}\right).
\end{align*}
The second inequality uses $\tau\le1$, and the last follows from
\eqref{eq:app-initial-bound} and the definition of $C_\star$.
The output rule gives
\[
\E\!\left[\|\nabla\Phi(x_R)\|^2\mid\mathcal F_K\right]
=\frac1K\sum_{k=0}^{K-1}\|\nabla\Phi(x_k)\|^2.
\]
Taking total expectations proves the finite-time output bound
\eqref{eq:main-finite-time} for every $0<\tau\le1$ and all positive
integers $m,K$.

Now use the parameters in \eqref{eq:main-parameters}.  They satisfy
$0<\tau\le1$ and
\begin{align*}
\frac{C_\star\Delta_0}{K}
 &\le\frac{\Delta_0}{3(1+\Delta_0)}\eps^2\le\frac{\eps^2}{3},\\
C_\star\tau^{2p}
 &=\frac{\eps^2}{3},\\
\frac{C_\star}{m\tau^2}
 &\le\frac{\eps^2}{3}.
\end{align*}
The conditions $C_\star\ge1$ and $1+\Delta_0>0$ also cover zero variance,
zero finite-difference bias, or zero initial gap.
Adding the three bounds yields
$\E\|\nabla\Phi(x_R)\|^2\le\eps^2$.

It remains to count primitive queries.  For fixed $p$, problem parameters,
and $\Delta_0$, the definitions give
\[
\tau^{-2}=(3C_\star)^{1/p}\eps^{-2/p},\qquad
m=\cO_p(\eps^{-2-2/p}),\qquad K=\cO_p(\eps^{-2}).
\]
Ceilings preserve these bounds for $0<\eps\le1$.
Each of the two finite differences uses at most $(p+1)m$ calls, and the
three center batches use $3m$ more.  No samples are shared between these
groups.  The total number is therefore at most
$(2p+5)mK=\cO_p(\eps^{-4-2/p})$.
At $p=1$ and $p=2$ the exponents are respectively $6$ and $5$.
\end{proof}

The constants are not optimized in the condition number or uniformly
over $p$.  Section~\ref{sec:lower-bound} combines this guarantee with a
matching lower bound for the same function class and oracle.

\section{Proof of Theorem~\ref{thm:lower}}
\label{app:lower-proof}

We construct $(A,u)$ in Sections~\ref{lb:sec:frontier}--\ref{lb:sec:scaling}
and prove the comparison-oracle bound in
Sections~\ref{lb:sec:random-rotation}--\ref{lb:sec:comparison}.
Lemma~\ref{lem:lb-gradient-variance} then gives the ordinary first-order
benchmark.  Section~\ref{lb:sec:lift} verifies the bilevel construction,
and Section~\ref{lb:sec:main-proof} combines its oracle simulation and
gradient estimate with the same comparison bound.

For a scalar function $h$, $\|h\|_\infty:=\sup_w|h(w)|$; for a vector
$v$, $\|v\|_\infty:=\max_j|v_j|$.  Unsubscripted vector norms are Euclidean.
The seed $\xi_0$ includes all private randomness of the algorithm and is
independent of the instance and oracle randomness.  The transcript
$\mathcal T_N$ contains the first $N$ queries and their responses.
All conditional laws are taken on standard Borel spaces.

\subsection{The chain and its decomposition}
\label{lb:sec:frontier}

\subsubsection{The Carmon--Arjevani chain}

For $v\in\R^T$ and $\alpha\ge0$, define
\[
  \prog_\alpha(v)
  :=\max\bigl(\{0\}\cup\{j\in[T]:|v_j|>\alpha\}\bigr).
\]
Let $v_0:=1$, and define
\[
  \Psi(t):=
  \begin{cases}
    0,&t\le\frac12,\\[1mm]
    \displaystyle
    \exp\!\left(1-\frac{1}{(2t-1)^2}\right),&t>\frac12,
  \end{cases}
\]
and
\[
  \varphi(t):=\sqrt e\int_{-\infty}^t e^{-s^2/2}\,ds.
\]
For $i\in[T]$, put
\begin{equation}\label{lb:eq:Hi}
  \psi_i(v)
  :=\Psi(-v_{i-1})\varphi(-v_i)
    -\Psi(v_{i-1})\varphi(v_i),
\end{equation}
and
\begin{equation}\label{lb:eq:FT}
  C_T(v):=\sum_{i=1}^T \psi_i(v).
\end{equation}
The convention $v_0=1$ makes
$\psi_1(v)=-\Psi(1)\varphi(v_1)$.

We use the following standard properties of this chain.  The constants below
are numerical and their precise values are immaterial for our scaling.

\begin{lemma}[Hard-chain properties; \cite{carmon2020stationary,arjevani2023lowerbounds}]
\label{lb:lem:hard-chain}
There exist numerical constants
$\Delta_{\rm c},L_0,G_\infty>0$ such that, for every $T$:
\begin{enumerate}[label=(\roman*)]
  \item $C_T(0)-\inf_v C_T(v)\le\Delta_{\rm c}T$;
  \item $\nabla C_T$ is $L_0$-Lipschitz;
  \item $\norm{\nabla C_T(v)}_\infty\le G_\infty$ for all $v$;
  \item
  $\prog_0(\nabla C_T(v))\le\prog_{1/2}(v)+1$;
  \item if $k:=\prog_1(v)<T$, then
  \begin{equation}\label{lb:eq:hard-chain-coordinate}
    \left|\partial_{k+1}C_T(v)\right|\ge1,
  \end{equation}
  and hence $\norm{\nabla C_T(v)}\ge1$.
\end{enumerate}
One may take $\Delta_{\rm c}=12$, $L_0=152$, and $G_\infty=23$.
\end{lemma}

\subsubsection{A smooth decomposition for the gradient oracle}

The decomposition $C_T=B_T-W_T$ separates the part of the gradient
that may expose the next coordinate.  We define its deterministic
weights below; randomness enters through \eqref{lb:eq:statistical-sample}.

Fix once and for all a nondecreasing function
$\Gamma\in C^\infty(\R;[0,1])$ such that
\begin{equation}\label{lb:eq:Gamma-flat}
  \Gamma(t)=0\quad(t\le1/4),
  \qquad
  \Gamma(t)=1\quad(t\ge1/2).
\end{equation}
Set
\[
  \gamma(t):=\Gamma(|t|).
\]
This is smooth because it vanishes on a neighborhood of the origin.  For
$i\in[T]$, define
\begin{equation}\label{lb:eq:ri-theta}
  r_i(v):=
  \left(\sum_{j=i}^T\gamma(v_j)^2\right)^{1/2},
  \qquad
  \Theta_i(v):=\Gamma(1-r_i(v)).
\end{equation}
The deterministic weight $\Theta_i$ equals one when the entire tail $v_{\ge i}$ is below the
$1/4$ progress threshold, and it equals zero once that tail contains a
coordinate of magnitude at least $1/2$.

Define the quiet base and frontier correction by
\begin{equation}\label{lb:eq:BT}
  B_T(v):=\sum_{i=1}^T \psi_i(v)(1-\Theta_i(v)),
\end{equation}
\begin{equation}\label{lb:eq:WT}
  W_T(v):=-\sum_{i=1}^T \psi_i(v)\Theta_i(v).
\end{equation}
Then, identically,
\begin{equation}\label{lb:eq:split}
  C_T=B_T-W_T.
\end{equation}
Equivalently, for any $\varpi\in(0,1]$, if
$\zeta\sim\mathrm{Bernoulli}(\varpi)$, the random scalar function
\begin{equation}\label{lb:eq:statistical-sample}
  h_T(v;\zeta):=B_T(v)-\frac{\zeta}{\varpi}W_T(v)
\end{equation}
satisfies $\E h_T(v;\zeta)=C_T(v)$.  This is the stochastic function construction of
\cite{arjevani2023lowerbounds}, written as a deterministic base and a
randomly retained correction.

For $k\in\{0,\ldots,T\}$, let $P_k$ denote the orthogonal projection onto
$\operatorname{span}\{e_1,\ldots,e_k\}$, with $P_0=0$ and
$P_{T+1}=I_T$.  For a $C^2$ scalar function $h$, write
\[
  \mathcal J_v^2h:=\bigl(h(v),\nabla h(v),\nabla^2h(v)\bigr).
\]

\begin{lemma}[Uniform frontier split]\label{lb:lem:frontier-split}
There exist numerical constants
$C_{W,0},C_{W,1},C_{W,2},L_B<\infty$, independent of $T$, such that:

\smallskip
\noindent\textup{(a) Uniform correction bounds.}
\begin{equation}\label{lb:eq:W-bounds}
  \norm{W_T}_\infty\le C_{W,0},
  \qquad
  \sup_v\norm{\nabla W_T(v)}\le C_{W,1},
  \qquad
  \sup_v\norm{\nabla^2W_T(v)}_{\op}\le C_{W,2},
\end{equation}
and
\begin{equation}\label{lb:eq:B-smooth}
  \sup_v\norm{\nabla^2B_T(v)}_{\op}\le L_B.
\end{equation}

\smallskip
\noindent\textup{(b) Single-frontier jet identity.}
If $r=\prog_{1/2}(v)$, then, with
$\psi_{T+1}\equiv\Theta_{T+1}\equiv0$,
\begin{equation}\label{lb:eq:single-frontier}
  \mathcal J_v^2W_T=-\mathcal J_v^2\bigl(\psi_{r+1}\Theta_{r+1}\bigr).
\end{equation}
Thus the complete second-order jet of $W_T$ is contributed by at most one
chain summand.

\smallskip
\noindent\textup{(c) Quiet-base truncation.}
If $k=\prog_{1/4}(v)$, then
\begin{align}
  B_T(v)&=B_T(P_k v),\label{lb:eq:B-value-trunc}\\
  \nabla B_T(v)&=P_k\nabla B_T(P_k v),\label{lb:eq:B-grad-trunc}\\
  \nabla^2B_T(v)&=P_k\nabla^2B_T(P_k v)P_k.\label{lb:eq:B-hess-trunc}
\end{align}
In particular,
\begin{equation}\label{lb:eq:B-support}
  \operatorname{supp}\nabla B_T(v)\subseteq\{1,\ldots,k\}.
\end{equation}

\smallskip
\noindent\textup{(d) One-step correction truncation.}
For the same $k$, the following identities hold; when $k=T$, they are
understood with the convention $P_{T+1}=I_T$ and are tautological:
\begin{align}
  W_T(v)&=W_T(P_{k+1}v),\label{lb:eq:W-value-trunc}\\
  \nabla W_T(v)&=P_{k+1}\nabla W_T(P_{k+1}v),\label{lb:eq:W-grad-trunc}\\
  \nabla^2W_T(v)&=P_{k+1}\nabla^2W_T(P_{k+1}v)P_{k+1}.\label{lb:eq:W-hess-trunc}
\end{align}
Consequently,
\begin{equation}\label{lb:eq:W-support}
  \operatorname{supp}\nabla W_T(v)
  \subseteq\{1,\ldots,\min\{k+1,T\}\}.
\end{equation}

\smallskip
\noindent\textup{(e) Removal of the next-coordinate term from the base.}
If $k<T$, then
\begin{equation}\label{lb:eq:theta-next-flat}
  \Theta_{k+1}(v)=1,
  \qquad
  \nabla\Theta_{k+1}(v)=0,
  \qquad
  \nabla^2\Theta_{k+1}(v)=0,
\end{equation}
and hence
\begin{equation}\label{lb:eq:next-link-deleted}
  \mathcal J_v^2\bigl(\psi_{k+1}(1-\Theta_{k+1})\bigr)=0.
\end{equation}
\end{lemma}

\begin{proof}
We divide the proof into four steps.

\paragraph{Step 1: dimension-free derivative bounds for the weights.}
Let
\[
  {C_{\Gamma,1}}:=\norm{\Gamma'}_\infty,
  \qquad
  {C_{\Gamma,2}}:=\norm{\Gamma''}_\infty.
\]
Fix $i$ and set $w_j(v):=\gamma(v_j)$ for $j=i,\ldots,T$.  Write
\[
  w(v):=(w_i(v),\ldots,w_T(v)),
  \qquad
  r(v):=\norm{w(v)}_2.
\]
Then $\Theta_i(v)=\Gamma(1-r(v))$.  The composition is smooth at
$r=0$ because it is constant whenever $r\le1/2$.  Its first and second
derivatives vanish for $r\le1/2$ and for $r\ge3/4$; at the two endpoints
this follows from the smooth flatness of $\Gamma$, rather than from a
two-sided neighborhood of constancy.  Therefore derivatives can be nonzero
only on
\begin{equation}\label{lb:eq:r-transition}
  \frac12<r<\frac34.
\end{equation}
On this region,
\[
  \nabla r=Dw(v)^\top\frac{w(v)}{r},
  \qquad
  \norm{\nabla r}\le {C_{\Gamma,1}}.
\]
Moreover,
\[
  \nabla^2r
  =\frac1rDw(v)^\top
  \left(I-\frac{w(v)w(v)^\top}{r^2}\right)Dw(v)
  +\sum_{j=i}^T\frac{w_j(v)}{r}\nabla^2w_j(v).
\]
The first term has operator norm at most $2{C_{\Gamma,1}}^2$ by
\eqref{lb:eq:r-transition}.  The second term is diagonal, and its operator norm
is at most ${C_{\Gamma,2}}$.  Hence
\begin{equation}\label{lb:eq:r-derivative-bounds}
  \norm{\nabla r}\le {C_{\Gamma,1}},
  \qquad
  \norm{\nabla^2r}_{\op}\le2{C_{\Gamma,1}}^2+{C_{\Gamma,2}}
\end{equation}
uniformly in $T$ and $i$ on the only region where the derivatives of
$\Theta_i$ can be nonzero.  The chain rule gives numerical constants
$C_{\Theta,1},C_{\Theta,2}$ such that
\begin{equation}\label{lb:eq:theta-derivative-bounds}
  \sup_{T,i,v}\norm{\nabla\Theta_i(v)}\le C_{\Theta,1},
  \qquad
  \sup_{T,i,v}\norm{\nabla^2\Theta_i(v)}_{\op}
  \le C_{\Theta,2}.
\end{equation}

\paragraph{Step 2: only one product has a nonzero second-order jet.}
Fix $v$ and let $r=\prog_{1/2}(v)$.  If $i\le r$, the tail $v_{\ge i}$
contains $v_r$ with $|v_r|>1/2$.  Thus $\gamma(v_r)=1$, so
$r_i(v)\ge1$ and $1-r_i(v)\le0$.  Since $\Gamma$ is identically zero on
$(-\infty,1/4]$,
\begin{equation}\label{lb:eq:early-mask-zero}
  \mathcal J_v^2\Theta_i=0
  \qquad(i\le r).
\end{equation}
If $i\ge r+2$, then $|v_{i-1}|\le1/2$.  Both $v_{i-1}$ and
$-v_{i-1}$ lie in the flat-zero region of $\Psi$, and therefore
\begin{equation}\label{lb:eq:late-link-zero}
  \mathcal J_v^2\psi_i=0
  \qquad(i\ge r+2).
\end{equation}
Combining \eqref{lb:eq:early-mask-zero} and \eqref{lb:eq:late-link-zero}, every
summand of $W_T$ has zero second-order jet except possibly the one indexed by
$r+1$.  This proves \eqref{lb:eq:single-frontier}.

\paragraph{Step 3: uniform bounds.}
Each $\psi_i$ is a translate of a fixed two-coordinate smooth template.  Hence
there are numerical constants $C_{\psi,0},C_{\psi,1},C_{\psi,2}$ such that
\[
  |\psi_i(v)|\le C_{\psi,0},
  \qquad
  \norm{\nabla \psi_i(v)}\le C_{\psi,1},
  \qquad
  \norm{\nabla^2\psi_i(v)}_{\op}\le C_{\psi,2}
\]
uniformly in $T,i,v$.  By the single-frontier identity,
\[
  |W_T(v)|\le C_{\psi,0},
\]
\[
  \norm{\nabla W_T(v)}
  \le C_{\psi,1}+C_{\psi,0}C_{\Theta,1},
\]
and
\[
  \norm{\nabla^2W_T(v)}_{\op}
  \le C_{\psi,2}+2C_{\psi,1}C_{\Theta,1}+C_{\psi,0}C_{\Theta,2}.
\]
This proves \eqref{lb:eq:W-bounds}.  Since $B_T=C_T+W_T$ and
Lemma~\ref{lb:lem:hard-chain} gives a dimension-free Hessian bound for $C_T$,
\eqref{lb:eq:B-smooth} follows.

\paragraph{Step 4: the truncation identities.}
Let $k=\prog_{1/4}(v)$.  For every $j>k$,
$|v_j|\le1/4$, and flatness of $\Gamma$ gives
\begin{equation}\label{lb:eq:tail-flat}
  \gamma(v_j)=\gamma'(v_j)=\gamma''(v_j)=0.
\end{equation}
For $i\le k$, the summand $\psi_i$ uses only $v_{i-1},v_i$, and all dependence of
$\Theta_i$ on coordinates after $k$ has zero second-order jet by
\eqref{lb:eq:tail-flat}.  If $k<T$, then $r_{k+1}(v)=0$, hence
$\Theta_{k+1}(v)=\Gamma(1)=1$ with zero first and second derivatives.  This
proves \eqref{lb:eq:theta-next-flat}--\eqref{lb:eq:next-link-deleted}.  For
$i\ge k+2$, the predecessor satisfies $|v_{i-1}|\le1/4<1/2$, so
$\mathcal J_v^2\psi_i=0$.  If $k=T$, there are no indices after $k$ and the truncation
identities are immediate.  Thus, in every case, the value, gradient, and
Hessian of $B_T$ are unchanged when all coordinates after $k$ are set to
zero, proving \eqref{lb:eq:B-value-trunc}--\eqref{lb:eq:B-hess-trunc}.

For $W_T$, if $k<T$, the only potentially active product has index $r+1$,
where $r=\prog_{1/2}(v)\le k$, and its chain summand depends on at most coordinates
$r,r+1$.  The projection $P_{k+1}$ deletes only coordinates $j>k+1$, all of
which satisfy \eqref{lb:eq:tail-flat}; in particular it leaves the frontier index
$r=\prog_{1/2}(v)$ unchanged.  On every deleted coordinate,
$\gamma,\gamma'$, and $\gamma''$ vanish, so evaluating at $P_{k+1}v$ selects
the same frontier product and preserves its complete second-order jet.  When
$k=T$, the claimed identities are tautological under $P_{T+1}=I_T$.  This
proves \eqref{lb:eq:W-value-trunc}--\eqref{lb:eq:W-hess-trunc}.
\end{proof}

\begin{corollary}[Two-vector probability-$\varpi$ zero-chain]
\label{lb:cor:two-channel-zero-chain}
For any $\varpi\in(0,1]$, let $\zeta\sim\mathrm{Bernoulli}(\varpi)$ and define
\[
  a^{(1)}(v,\zeta):=\nabla B_T(v),
  \qquad
  a^{(2)}(v,\zeta):=\frac{\zeta}{\varpi}\nabla W_T(v).
\]
Then, simultaneously for every $v$,
\[
  \max_{\ell\in\{1,2\}}\prog_0(a^{(\ell)}(v,\zeta))
  \le\prog_{1/4}(v)+1,
\]
and on the event $\{\zeta=0\}$,
\[
  \max_{\ell\in\{1,2\}}\prog_0(a^{(\ell)}(v,\zeta))
  \le\prog_{1/4}(v).
\]
The exceptional event is independent of $v$ and has probability $\varpi$.
\end{corollary}

\subsection{Construction and scaling of the smooth pair}
\label{lb:sec:scaling}

For $R>0$, define the soft projection
\begin{equation}\label{lb:eq:soft-projection}
  \rho_R(w):=\frac{w}{\sqrt{1+\norm w^2/R^2}},
\end{equation}
and let $J_R(w)=D\rho_R(w)$.  We use the standard choices
\begin{equation}\label{lb:eq:R-eta}
  R:=230\sqrt T,
  \qquad
  {\eta_{\rm r}}:=\frac15.
\end{equation}
Let $\Ortho(d,T):=\{U\in\R^{d\times T}:U^\top U=I_T\}$.
For $U\in\Ortho(d,T)$, define
\begin{align}
  {C_U^{\rm c}}(w)
  &:=C_T(U^\top\rho_R(w))+\frac{\eta_{\rm r}}2\norm w^2,
  \label{lb:eq:Fhat}\\
  {B_U^{\rm c}}(w)
  &:=B_T(U^\top\rho_R(w))+\frac{\eta_{\rm r}}2\norm w^2,
  \label{lb:eq:Bhat}\\
  {W_U^{\rm c}}(w)
  &:=W_T(U^\top\rho_R(w)).
  \label{lb:eq:What}
\end{align}
Then
\begin{equation}\label{lb:eq:compressed-split}
  {C_U^{\rm c}}={B_U^{\rm c}}-{W_U^{\rm c}}.
\end{equation}

\begin{lemma}[Compressed hard-chain properties]
\label{lb:lem:compressed-hard-chain}
There is a numerical constant $L_{\rm hc}$ such that, uniformly over
$T,d,U$:
\begin{enumerate}[label=(\roman*)]
  \item
  ${C_U^{\rm c}}(0)-\inf_w{C_U^{\rm c}}(w)\le\Delta_{\rm c}T$;
  \item $\nabla{C_U^{\rm c}}$ is $L_{\rm hc}$-Lipschitz;
  \item if
  \begin{equation}\label{lb:eq:compressed-progress-condition}
    \prog_{1/4}(U^\top\rho_R(w))<T,
  \end{equation}
  then
  \begin{equation}\label{lb:eq:compressed-gradient-lb}
    \norm{\nabla{C_U^{\rm c}}(w)}\ge\frac12;
  \end{equation}
  \item
  $\norm{J_R(w)}_{\op}\le1$ and
  \begin{equation}\label{lb:eq:J-Lipschitz}
    \norm{J_R(w)-J_R(w')}_{\op}
    \le\frac3R\norm{w-w'}.
  \end{equation}
\end{enumerate}
One may take $L_{\rm hc}=155$.
\end{lemma}

\begin{proof}
The compression estimates in Lemma~15 of \cite{arjevani2023lowerbounds} give
(iv) and show that $\rho_R$ is $1$-Lipschitz.  The nonnegative quadratic
term vanishes at zero, so (i) follows from Lemma~\ref{lb:lem:hard-chain}.
For (ii), the chain rule gives the gradient Lipschitz bound
\[
  L_0+\frac{3G_\infty\sqrt T}{R}+{\eta_{\rm r}}
  =152+\frac3{10}+\frac15<155.
\]
It remains to prove (iii).  Put
\[
  v:=\rho_R(w),\qquad h:=U^\top v,\qquad
  {v_g}:=U\nabla C_T(h),\qquad
  c:=\sqrt{1+\norm w^2/R^2}.
\]
Then $w=cv$, $\norm {v_g}\le G_\infty\sqrt T=23\sqrt T$, and
\begin{equation}\label{lb:eq:compressed-gradient-formula}
  \nabla{C_U^{\rm c}}(w)
  =\frac{I-vv^\top/R^2}{c}\,{v_g}+{\eta_{\rm r}} w.
\end{equation}

First suppose $\norm w\le R/2$.  Condition
\eqref{lb:eq:compressed-progress-condition} implies
$k:=\prog_1(h)<T$.  Set $j:=k+1$.  Then $|h_j|\le1$, while
Lemma~\ref{lb:lem:hard-chain} gives
$|\ip{u_j}{{v_g}}|=|\partial_jC_T(h)|\ge1$.  Since
$c\le\sqrt5/2$ and $\norm v\le R/2$, equation
\eqref{lb:eq:compressed-gradient-formula} yields
\begin{align*}
\norm{\nabla{C_U^{\rm c}}(w)}
&\ge\left|\ip{u_j}{\nabla{C_U^{\rm c}}(w)}\right|\\
&\ge \frac1c
 -\frac{|h_j|\,|\ip v {v_g}|}{R^2c}
 -{\eta_{\rm r}} c|h_j|\\
&\ge \frac{2}{\sqrt5}
 -\frac{23\sqrt T}{2R}
 -\frac{{\eta_{\rm r}}\sqrt5}{2}\\
&=\frac{2}{\sqrt5}-\frac1{20}-\frac{\sqrt5}{10}
>\frac12,
\end{align*}
where the equality uses $R=230\sqrt T$ and ${\eta_{\rm r}}=1/5$.

Now suppose $\norm w>R/2$.  Lemma~15 of \cite{arjevani2023lowerbounds} gives
$\norm{J_R(w)}_{\op}\le c^{-1}\le2/\sqrt5$, so
\begin{align*}
\norm{\nabla{C_U^{\rm c}}(w)}
&\ge {\eta_{\rm r}}\norm w-\norm{J_R(w)}_{\op}\norm {v_g}\\
&>\frac{R}{10}-\frac{46}{\sqrt5}\sqrt T\\
&=\left(23-\frac{46}{\sqrt5}\right)\sqrt T
>\frac12.
\end{align*}
This proves \eqref{lb:eq:compressed-gradient-lb}.  The numerical smoothness
constant $L_{\rm hc}=155$ is the corresponding bound obtained from the same
compression estimates.
\end{proof}

\begin{lemma}[Compressed correction bounds]
\label{lb:lem:compressed-correction}
There are numerical constants
${C_{W,0}^{\rm c}},{C_{W,1}^{\rm c}},{C_{W,2}^{\rm c}}$ such that, uniformly in
$T,d,U$,
\begin{equation}\label{lb:eq:What-bounds}
  \norm{{W_U^{\rm c}}}_\infty\le{C_{W,0}^{\rm c}},
  \qquad
  \sup_w\norm{\nabla{W_U^{\rm c}}(w)}\le{C_{W,1}^{\rm c}},
\end{equation}
\begin{equation}\label{lb:eq:What-hessian-bound}
  \sup_w\norm{\nabla^2{W_U^{\rm c}}(w)}_{\op}
  \le{C_{W,2}^{\rm c}}.
\end{equation}
Consequently, $\nabla{B_U^{\rm c}}$ is
$(L_{\rm hc}+{C_{W,2}^{\rm c}})$-Lipschitz.
\end{lemma}

\begin{proof}
The value bound follows from Lemma~\ref{lb:lem:frontier-split}.  By the chain
rule,
\[
  \nabla{W_U^{\rm c}}(w)
  =J_R(w)^\top U\nabla W_T(U^\top\rho_R(w)),
\]
so the gradient bound follows from $\norm{J_R(w)}_{\op}\le1$.  For $w,w'$, use
Lemma~\ref{lb:lem:frontier-split}, the $1$-Lipschitzness of $\rho_R$, and
\eqref{lb:eq:J-Lipschitz} to obtain
\begin{align*}
&\norm{\nabla{W_U^{\rm c}}(w)
      -\nabla{W_U^{\rm c}}(w')}\\
&\quad\le
\frac{3C_{W,1}}R\norm{w-w'}
+C_{W,2}\norm{w-w'}.
\end{align*}
Since $R\ge1$ for $T\ge1$, this gives a numerical Hessian bound.  Finally,
${B_U^{\rm c}}={C_U^{\rm c}}+{W_U^{\rm c}}$.
\end{proof}

Set
\begin{equation}\label{lb:eq:ell-star}
  \ell_\star
  :=L_{\rm hc}+{C_{W,2}^{\rm c}}+1.
\end{equation}
For sufficiently small $\eps$, choose
\begin{equation}\label{lb:eq:T-choice}
  T:=\left\lfloor c_T\eps^{-2}\right\rfloor,
  \qquad
  c_T:=\frac{1}{64\Delta_{\rm c}\ell_\star},
\end{equation}
so $T\ge4$.  Define
\begin{equation}\label{lb:eq:lambda-a}
  {\lambda_{\rm lb}}:=4\ell_\star\eps,
  \qquad
  a:=\frac{{\lambda_{\rm lb}}^2}{\ell_\star}=16\ell_\star\eps^2.
\end{equation}
Define the base, correction, and nonconvex objective by
\begin{align}
  {A_U}(x)
  &:=a{B_U^{\rm c}}(x/{\lambda_{\rm lb}}),
  \label{lb:eq:scaled-Phi}\\
  u_U(x)
  &:=a{W_U^{\rm c}}(x/{\lambda_{\rm lb}}),
  \label{lb:eq:scaled-u}\\
  P_U(x)
  &:={A_U}(x)-u_U(x)
   =a{C_U^{\rm c}}(x/{\lambda_{\rm lb}}).
  \label{lb:eq:scaled-F}
\end{align}

\begin{proposition}[Properties of the scaled pair]
\label{lb:prop:scaled-pair}
There are numerical constants $C_{u,0},C_{u,1},C_{u,2}$ such that, uniformly
in $U,T,\eps$,
\begin{equation}\label{lb:eq:scaled-u-bounds}
  \norm{u_U}_\infty\le C_{u,0}\eps^2,
  \qquad
  \sup_x\norm{\nabla u_U(x)}\le C_{u,1}\eps,
  \qquad
  \sup_x\norm{\nabla^2u_U(x)}_{\op}\le C_{u,2}.
\end{equation}
Moreover,
\begin{equation}\label{lb:eq:Phi-hessian-bound}
  \sup_x\norm{\nabla^2{A_U}(x)}_{\op}\le1,
\end{equation}
\begin{equation}\label{lb:eq:scaled-gap}
  P_U(0)-\inf_xP_U(x)
  \le\frac14,
\end{equation}
and
\begin{equation}\label{lb:eq:scaled-stationarity}
  \prog_{1/4}\!\left(U^\top\rho_R(x/{\lambda_{\rm lb}})\right)<T
  \quad\Longrightarrow\quad
  \norm{\nabla P_U(x)}\ge2\eps.
\end{equation}
\end{proposition}

\begin{proof}
The three correction scales follow from
Lemma~\ref{lb:lem:compressed-correction} and
\[
  a=16\ell_\star\eps^2,
  \qquad
  \frac a{\lambda_{\rm lb}}=4\eps,
  \qquad
  \frac a{{\lambda_{\rm lb}}^2}=\frac1{\ell_\star}.
\]
The same identities and the definition of $\ell_\star$ give
\eqref{lb:eq:Phi-hessian-bound}.  By
Lemma~\ref{lb:lem:compressed-hard-chain} and \eqref{lb:eq:T-choice},
\[
  P_U(0)-\inf P_U
  \le a\Delta_{\rm c}T
  \le16\ell_\star\eps^2\Delta_{\rm c}c_T\eps^{-2}
  =\frac14.
\]
Finally,
\[
  \nabla P_U(x)
  =\frac a{\lambda_{\rm lb}}\nabla{C_U^{\rm c}}(x/{\lambda_{\rm lb}})
  =4\eps\nabla{C_U^{\rm c}}(x/{\lambda_{\rm lb}}),
\]
and \eqref{lb:eq:scaled-stationarity} follows from
Lemma~\ref{lb:lem:compressed-hard-chain}.
\end{proof}

\subsection{Finite-channel random rotation}
\label{lb:sec:random-rotation}

The comparison oracle returns two ambient vectors and a scalar.
We prove a random-rotation bound that permits these separate channels,
including the scalar information, and applies to arbitrary adaptive
queries and final outputs.

\begin{definition}[Uniform finite-channel probability-$\varpi$ zero-chain]
\label{lb:def:finite-chain}
Fix $\varpi\in[0,1]$.  Let $\mathcal Y,\Xi,\mathcal S$ be standard Borel spaces, let
$\xi\sim\nu$, and let
\[
  s:\R^T\times\mathcal Y\times\Xi\to\mathcal S,
\]
\[
  a^{(\ell)}:\R^T\times\mathcal Y\times\Xi\to\R^T,
  \qquad \ell\in[m],
\]
be Borel maps.  We call $(s,a^{(1)},\ldots,a^{(m)})$ a uniform finite-channel
probability-$\varpi$ zero-chain if there is a measurable event
$E\subseteq\Xi$ with $\nu(E)\le \varpi$ such that, for every $(v,y,\xi)$,
\begin{equation}\label{lb:eq:finite-chain-one-step}
  \max_{\ell\in[m]}\prog_0(a^{(\ell)}(v,y,\xi))
  \le\prog_{1/4}(v)+1,
\end{equation}
and, whenever $\xi\notin E$,
\begin{equation}\label{lb:eq:finite-chain-quiet}
  \max_{\ell\in[m]}\prog_0(a^{(\ell)}(v,y,\xi))
  \le\prog_{1/4}(v).
\end{equation}
The map $s$ collects all non-ambient channels.  It may depend on the hidden
rotation only through the latent query $v$.
\end{definition}

For $U\in\Ortho(d,T)$, define the canonical rotated response
\begin{equation}\label{lb:eq:canonical-rotated-oracle}
  \widetilde{\cO}_U(x,y;\xi)
  :=\left(
  s(U^\top x,y,\xi),
  Ua^{(1)}(U^\top x,y,\xi),\ldots,
  Ua^{(m)}(U^\top x,y,\xi)
  \right).
\end{equation}

\begin{lemma}[Projection accumulation; cf. Lemma 12 of \cite{arjevani2023lowerbounds}]
\label{lb:lem:projection-accumulation}
Fix integers $n,T\ge1$, orthonormal vectors $u_1,\ldots,u_T\in\R^d$,
and vectors $x_1,\ldots,x_n\in\R^d$ with $\norm{x_i}\le R$.  For
$i\in\{0,\ldots,n\}$ and $j\in\{0,\ldots,T\}$, let
\[
  \mathsf S_j^{(i)}
  :=\operatorname{span}\{u_1,\ldots,u_j,x_1,\ldots,x_i\},
\]
let $P_j^{(i)}$ be the orthogonal projection onto
$(\mathsf S_j^{(i)})^\perp$, and put
$\Pi_j^{(i)}:=I-P_j^{(i)}$.  For $i\in[n]$ and $j\in[T]$, define
\begin{equation}\label{lb:eq:Gij}
  \mathcal G_j^{(i)}
  :=\left\{
  \norm{\Pi_{j-1}^{(i)}P_{j-1}^{(i-1)}u_j}
  <\frac{1}{4R\sqrt n}
  \right\}.
\end{equation}
Then, for every $i\in[n]$ and $j\in[T]$,
\begin{equation}\label{lb:eq:projection-accumulation-conclusion}
  \bigcap_{q=1}^i\mathcal G_j^{(q)}
  \quad\Longrightarrow\quad
  |\ip{u_j}{x_i}|
  \le\frac{\norm{x_i}}{4R}\sqrt{\frac in}
  \le\frac14.
\end{equation}
\end{lemma}

\begin{proof}
For fixed $j$, the spaces $\mathsf S_{j-1}^{(q)}$ are nested in $q$.
Consequently,
\[
  \Pi_{j-1}^{(q-1)}
  +P_{j-1}^{(q-1)}\Pi_{j-1}^{(q)}P_{j-1}^{(q-1)}
  =\Pi_{j-1}^{(q)}.
\]
Iterating this identity gives
\begin{equation}\label{lb:eq:projection-telescope}
  \Pi_{j-1}^{(i)}
  =\sum_{r=1}^{j-1}u_ru_r^\top
  +\sum_{q=1}^i
  P_{j-1}^{(q-1)}\Pi_{j-1}^{(q)}P_{j-1}^{(q-1)}.
\end{equation}
Since $x_i\in\mathsf S_{j-1}^{(i)}$ and $u_j\perp u_r$ for $r<j$,
\eqref{lb:eq:projection-telescope} and Cauchy--Schwarz yield
\begin{align}
  |\ip{u_j}{x_i}|
  &\le
  \left(
  \sum_{q=1}^i
  \norm{\Pi_{j-1}^{(q)}P_{j-1}^{(q-1)}u_j}^2
  \right)^{1/2}
  \nonumber\\
  &\qquad\times
  \left(
  \sum_{q=1}^i
  \norm{\Pi_{j-1}^{(q)}P_{j-1}^{(q-1)}x_i}^2
  \right)^{1/2}.
  \label{lb:eq:projection-cs}
\end{align}
On $\cap_{q=1}^i\mathcal G_j^{(q)}$, the first squared factor is less
than $i/(16R^2n)$.  Moreover, subtracting the first term on the right of
\eqref{lb:eq:projection-telescope} shows that
\[
  \sum_{q=1}^i
  P_{j-1}^{(q-1)}\Pi_{j-1}^{(q)}P_{j-1}^{(q-1)}
  \preceq I.
\]
The second squared factor in \eqref{lb:eq:projection-cs} is therefore at most
$\norm{x_i}^2$.  This proves \eqref{lb:eq:projection-accumulation-conclusion}.
\end{proof}

For the remaining geometric lemmas, fix an adaptive algorithm and a horizon
$n$.  Assume that $U$ has the uniform invariant Stiefel distribution on
$\Ortho(d,T)$, independently of the
algorithm seed and the i.i.d. oracle seeds, and write
$U=[u_1,\ldots,u_T]$ and
\[
  v_i:=U^\top x_i,
  \qquad
  a_i^{(\ell)}:=a^{(\ell)}(v_i,y_i,\xi_i),
\]
\[
  \gamma_i
  :=\max_{\substack{1\le q\le i\\1\le\ell\le m}}
  \prog_0(a_q^{(\ell)}),
  \qquad
  \gamma_0:=0.
\]
Thus $\gamma_i$ is the largest hidden coordinate index appearing in a
canonical vector response by call $i$.  It is a proof bookkeeping
quantity; the algorithm need not identify the individual hidden basis
vectors or observe $\gamma_i$.
For $i\in\{0,\ldots,n\}$ and $j\in\{0,\ldots,T\}$, define
\begin{equation}\label{lb:eq:Uji-sigma-field}
  \mathcal U_j^{(i)}
  :=\sigma\!\left(
  \xi_0,u_1,\ldots,u_j,
  v_1,\xi_1,\ldots,v_i,\xi_i
  \right),
\end{equation}
where the list after $u_j$ is empty for $i=0$, and define
\begin{equation}\label{lb:eq:Tji-event}
  \mathcal T_j^{(i)}:=\{j>\gamma_i\}.
\end{equation}
Crucially, $\mathcal U_{j-1}^{(i-1)}$ contains only the \emph{past} latent
queries $v_1,\ldots,v_{i-1}$; it does not contain the current latent query
$v_i$.  All conditional laws below are regular conditional laws, whose
existence follows from the standard-Borel hypotheses.

\begin{lemma}[Invariant posterior under an invariant observation]
\label{lb:lem:invariant-posterior}
Let $\mathsf C,\mathsf X,\mathsf Z$ be standard Borel spaces, let
$G$ be a compact second-countable group with normalized Haar probability
$\lambda_G$, and let
\[
  \tau:\mathsf C\times G\times\mathsf X\to\mathsf X
\]
be a Borel map such that $g\mapsto\tau_{c,g}$ is a group action for every
$c$.  Let $K(c,dx)$ be a Borel probability kernel satisfying
\begin{equation}\label{lb:eq:prior-kernel-invariance}
  K(c,\tau_{c,g}^{-1}A)=K(c,A)
  \qquad(c\in\mathsf C,\ g\in G)
\end{equation}
for every Borel $A\subseteq\mathsf X$.  Suppose that the Borel observation
$H:\mathsf C\times\mathsf X\to\mathsf Z$ is invariant:
\begin{equation}\label{lb:eq:observation-invariance}
  H(c,\tau_{c,g}x)=H(c,x).
\end{equation}
Then there is a Borel regular conditional kernel
$L(c,v,dx)$ for $x$ given $v=H(c,x)$ under $K(c,dx)$ such that
\begin{equation}\label{lb:eq:posterior-kernel-invariance}
  L(c,v,\tau_{c,g}^{-1}A)=L(c,v,A)
\end{equation}
for every $g\in G$ and for $H(c,\cdot)_\#K(c)$-almost every $v$.
\end{lemma}

\begin{proof}
Let $L_0(c,v,dx)$ be any Borel disintegration of $K(c,dx)$ over the map
$H(c,\cdot)$, and define
\begin{equation}\label{lb:eq:averaged-posterior}
  L(c,v,A)
  :=\int_G L_0(c,v,\tau_{c,g}^{-1}A)\,\lambda_G(dg).
\end{equation}
This is a Borel kernel.  If
$M(c,dv):=H(c,\cdot)_\#K(c)$, then for Borel
$A\subseteq\mathsf X$ and $B\subseteq\mathsf Z$,
\begin{align*}
\int_B L(c,v,A)M(c,dv)
&=\int_G K\bigl(c,\tau_{c,g}^{-1}A\cap H(c,\cdot)^{-1}B\bigr)
  \,\lambda_G(dg)\\
&=\int_G K\bigl(c,\tau_{c,g}^{-1}
  (A\cap H(c,\cdot)^{-1}B)\bigr)\,\lambda_G(dg)\\
&=K\bigl(c,A\cap H(c,\cdot)^{-1}B\bigr),
\end{align*}
where the second line uses \eqref{lb:eq:observation-invariance} and the third
uses \eqref{lb:eq:prior-kernel-invariance}.  Hence $L$ is a regular
conditional kernel.  Equation \eqref{lb:eq:posterior-kernel-invariance}
follows directly from \eqref{lb:eq:averaged-posterior} and left/right
invariance of Haar probability on the compact group $G$.
\end{proof}

\begin{lemma}[Sequential posterior-stabilizer invariance]
\label{lb:lem:sequential-posterior}
Let $U=[u_1,\ldots,u_T]$ have the uniform invariant Stiefel distribution on
$\Ortho(d,T)$, and let $\vartheta$
be an independent random element of a standard Borel space.  Fix
$j\in[T]$ and an integer $r\ge0$.  Suppose that Borel maps $\chi_q$ define
recursively
\begin{equation}\label{lb:eq:sequential-query-recursion}
  x_q=\chi_q(\vartheta,u_1,\ldots,u_{j-1},v_1,\ldots,v_{q-1}),
  \qquad
  v_q=U^\top x_q,
  \qquad q\in[r].
\end{equation}
Let
\[
  C_q:=(\vartheta,u_1,\ldots,u_{j-1},v_1,\ldots,v_q),
\]
with $C_0=(\vartheta,u_1,\ldots,u_{j-1})$, and set
\[
  \mathsf S_q
  :=\operatorname{span}\{u_1,\ldots,u_{j-1},x_1,\ldots,x_q\}.
\]
Regular conditional kernels $K_q(C_q,dU)$ of $U$ given $C_q$ may be
chosen so that, for almost every $C_q=c_q$, the kernel is invariant under
every orthogonal map fixing $\mathsf S_q(c_q)$ pointwise:
\begin{equation}\label{lb:eq:sequential-posterior-invariance}
  K_q(c_q,W^{-1}A)=K_q(c_q,A)
  \quad
  \text{whenever }Ws=s\text{ for all }s\in\mathsf S_q(c_q).
\end{equation}
\end{lemma}

\begin{proof}
We construct the kernels inductively.  Conditional on
$u_1,\ldots,u_{j-1}$, the remaining columns of a uniform Stiefel frame have
the standard invariant-completion law.  Explicitly, after choosing a Borel
orthonormal frame $Q_\perp(u_{<j})$ for
$\operatorname{span}(u_1,\ldots,u_{j-1})^\perp$, this law is the image of
the uniform invariant probability measure on the Stiefel manifold
$\Ortho(d-j+1,T-j+1)$ under
\[
  V\longmapsto [u_1,\ldots,u_{j-1},Q_\perp(u_{<j})V].
\]
It is invariant under every orthogonal map fixing
$\mathsf S_0=\operatorname{span}\{u_1,\ldots,u_{j-1}\}$ pointwise.
Independence of $\vartheta$ gives the required kernel $K_0$.

Assume that $K_{q-1}$ has been chosen.  Let
$\mathsf C_{q-1}^{\rm good}$ be a Borel full-measure subset of the
conditioning space on which the induction hypothesis holds.  We first
restrict the conditioning parameter and the kernel to this good set; only
after constructing the next posterior there will we extend it arbitrarily
over the null complement.  Once $C_{q-1}=c_{q-1}\in
\mathsf C_{q-1}^{\rm good}$ is fixed, the vector $x_q$ in
\eqref{lb:eq:sequential-query-recursion} is fixed.  Put
\[
  \mathsf S_q
  =\operatorname{span}(\mathsf S_{q-1},x_q).
\]
On each stratum on which $\dim\mathsf S_q=s$, choose a Borel orthogonal
completion $Q_\perp(c_{q-1})$ of $\mathsf S_q$.  For
$R\in O(d-s)$ define
\[
  W_{c_{q-1},R}
  :=P_{\mathsf S_q}
    +Q_\perp(c_{q-1})R Q_\perp(c_{q-1})^\top.
\]
This is a Borel action of the fixed compact group $O(d-s)$ on the
conditional Stiefel fiber.  Because it fixes $\mathsf S_q$ pointwise, it
belongs to the stabilizer from the induction hypothesis.  Moreover, the
new observation
\[
  H_q(c_{q-1},U):=U^\top x_q
\]
is invariant under this action:
\[
  (W_{c_{q-1},R}U)^\top x_q
  =U^\top W_{c_{q-1},R}^\top x_q
  =U^\top x_q.
\]
Apply Lemma~\ref{lb:lem:invariant-posterior}, separately on each Borel rank
stratum of $\mathsf C_{q-1}^{\rm good}$, to disintegrate $K_{q-1}$ over
$H_q$.  The finitely many resulting kernels paste into a Borel kernel on the
good set; extend that kernel arbitrarily over the null complement.  The
resulting $K_q$ is a regular conditional law given $C_q$ and, for almost
every $C_q$, is invariant under the full pointwise stabilizer of
$\mathsf S_q$.  This closes the induction.
Every orthogonal transformation used in this argument is fixed only
after $C_{q-1}$ has been conditioned upon; no state-dependent global
change of variables is invoked.
\end{proof}

\begin{lemma}[Conditional residual spherical law; cf. Lemma 14 of
\cite{arjevani2023lowerbounds}]
\label{lb:lem:conditional-haar}
Fix $i\in[n]$ and $j\in[T]$.  On the event
$\mathcal T_j^{(i-1)}$, the past transcript and the current query
$(x_i,y_i)$ are measurable from $\mathcal U_{j-1}^{(i-1)}$.  In particular,
$\mathcal T_j^{(i-1)}$, the space
\[
  \mathsf S
  :=\operatorname{span}\{u_1,\ldots,u_{j-1},x_1,\ldots,x_{i-1}\},
\]
and its orthogonal projector are measurable under this conditioning.
There is an $\mathcal U_{j-1}^{(i-1)}$-measurable radius
$r_{i,j}\in[0,1]$ such that, outside a null event and on
$\mathcal T_j^{(i-1)}$,
\begin{equation}\label{lb:eq:conditional-haar-law}
\mathcal L\!\left(
P_{\mathsf S^\perp}u_j
\,\middle|\,
\mathcal U_{j-1}^{(i-1)}
\right)
=
\begin{cases}
\delta_0,&r_{i,j}=0,\\[1mm]
\operatorname{Unif}\!\left(
\{v\in\mathsf S^\perp:\norm v=r_{i,j}\}
\right),&r_{i,j}>0.
\end{cases}
\end{equation}
Moreover,
$r_{i,j}=\norm{P_{\mathsf S^\perp}u_j}$ almost surely.  Thus normalization
is used only on $\{r_{i,j}>0\}$.
\end{lemma}

\begin{proof}
We first construct globally Borel quiet-simulated queries.  For fixed
$u_{<j}:=(u_1,\ldots,u_{j-1})$, define the truncated canonical response
\begin{align}
\overline{\mathcal O}_{<j}(v,y;\xi,u_{<j})
:=\biggl(&s(v,y,\xi),
\sum_{r<j}[a^{(1)}(v,y,\xi)]_r u_r,\ldots,
\sum_{r<j}[a^{(m)}(v,y,\xi)]_r u_r\biggr).
\label{lb:eq:quiet-truncated-oracle}
\end{align}
Feed these responses recursively to the algorithm.  Since all algorithm
and oracle maps are Borel, there are Borel maps $\chi_q$ such that, for
$\vartheta:=(\xi_0,\xi_1,\ldots,\xi_{i-1})$,
\begin{equation}\label{lb:eq:quiet-query-maps}
  (\bar x_q,\bar y_q)
  =\chi_q(\vartheta,u_{<j},\bar v_1,\ldots,\bar v_{q-1}),
  \qquad
  \bar v_q:=U^\top\bar x_q,
  \qquad q\le i.
\end{equation}
For $q<i$, let $\bar\gamma_q$ be the progress computed from the
corresponding latent coefficients
$a^{(\ell)}(\bar v_q,\bar y_q,\xi_q)$.

A first-disagreement induction gives
\begin{equation}\label{lb:eq:quiet-actual-identification}
  \{j>\gamma_{i-1}\}=\{j>\bar\gamma_{i-1}\},
  \qquad
  (x_q,y_q,v_q)=(\bar x_q,\bar y_q,\bar v_q)
  \quad(q\le i)
\end{equation}
on this common event.  For the first implication, suppose
$j>\gamma_{i-1}$.  Then every actual latent coefficient before time $i$
is supported in $\{1,\ldots,j-1\}$.  Starting from the common first query,
the actual ambient response therefore equals
\eqref{lb:eq:quiet-truncated-oracle}; induction gives equality of the two
trajectories and hence $\bar\gamma_{i-1}=\gamma_{i-1}<j$.  The reverse
implication is identical with the roles of the actual and quiet trajectories
interchanged.

For completeness, measurability from the actual conditioning tuple can be
checked without referring to the random quiet trajectory.  Given
$(\xi_0,u_{<j},v_1,\xi_1,\ldots,v_{i-1},\xi_{i-1})$, recursively compute the
query, the non-ambient channel, and the latent coefficients at each past
round.  If a coefficient first has support reaching $j$, declare
$\mathcal T_j^{(i-1)}$ false; while all coefficients remain supported below
$j$, every ambient response is computable from $u_{<j}$ and the recursion
continues.  The declaration is correct by induction and, on the true event,
also reconstructs $(x_i,y_i)$.  Hence $\mathcal T_j^{(i-1)}$ and the current
query are $\mathcal U_{j-1}^{(i-1)}$-measurable.

Apply Lemma~\ref{lb:lem:sequential-posterior} to the quiet recursion
\eqref{lb:eq:quiet-query-maps}, with $r=i-1$.  It yields a regular
conditional kernel of $U$ given
\[
  (\vartheta,u_1,\ldots,u_{j-1},\bar v_1,\ldots,\bar v_{i-1}).
\]
Transport this kernel through the fixed Borel coordinate reordering and
write the conditioning tuple, in exactly the order used in
\eqref{lb:eq:Uji-sigma-field}, as
\[
  \bar C
  :=(\xi_0,u_1,\ldots,u_{j-1},
      \bar v_1,\xi_1,\ldots,\bar v_{i-1},\xi_{i-1}).
\]
Denote the transported kernel by $\bar K(\bar C,dU)$.  It is invariant
under every orthogonal map fixing
\[
  \bar{\mathsf S}
  :=\operatorname{span}\{u_1,\ldots,u_{j-1},
  \bar x_1,\ldots,\bar x_{i-1}\}
\]
pointwise.  Let
\[
  C:=(\xi_0,u_1,\ldots,u_{j-1},
      v_1,\xi_1,\ldots,v_{i-1},\xi_{i-1}),
\]
the tuple generating $\mathcal U_{j-1}^{(i-1)}$.  By
\eqref{lb:eq:quiet-actual-identification}, on $\mathcal T_j^{(i-1)}$ we now
have the literal equality $\bar C=C$, and
$\bar{\mathsf S}=\mathsf S$.  Write
$\bar K\phi(c):=\int\phi(U)\bar K(c,dU)$.  Because the events and the
conditioning tuples agree in \eqref{lb:eq:quiet-actual-identification}, for
every bounded Borel $h$ and $\phi$,
\begin{align*}
\E\!\left[
  \ind_{\mathcal T_j^{(i-1)}}h(C)\phi(U)
\right]
&=\E\!\left[
  \ind_{\{j>\bar\gamma_{i-1}\}}h(\bar C)\phi(U)
\right]\\
&=\E\!\left[
  \ind_{\{j>\bar\gamma_{i-1}\}}h(\bar C)\bar K\phi(\bar C)
\right]\\
&=\E\!\left[
  \ind_{\mathcal T_j^{(i-1)}}h(C)\bar K\phi(C)
\right].
\end{align*}
Thus $\bar K(C,\cdot)$ is a version of the actual conditional law on
$\mathcal T_j^{(i-1)}$.  We have therefore obtained, by a sequential
posterior-kernel induction, the conditional invariance of $U$ under the
pointwise stabilizer of $\mathsf S$.  In particular, the conditional law of
$u_j$ is invariant under that stabilizer.

It remains to identify its orbit.  The conditioning fixes all inner
products of $u_j$ with the spanning vectors of $\mathsf S$: they are zero
against $u_1,\ldots,u_{j-1}$, and
\[
  \ip{u_j}{x_q}=[v_q]_j,
  \qquad q<i.
\]
These values determine $b:=P_{\mathsf S}u_j$ uniquely, even when the
spanning vectors are linearly dependent.  Since $\norm{u_j}=1$, they also
determine the Borel radius
\[
  r_{i,j}:=\sqrt{1-\norm{b}^2}
  =\norm{P_{\mathsf S^\perp}u_j}.
\]
The residual conditional law is supported on the radius-$r_{i,j}$ sphere
in $\mathsf S^\perp$ and is invariant under its full orthogonal group.
For $r_{i,j}>0$, uniqueness of rotation-invariant spherical probability gives the
second case of \eqref{lb:eq:conditional-haar-law}; for $r_{i,j}=0$, the
sphere is the singleton $\{0\}$ and the law is $\delta_0$.
\end{proof}

\begin{lemma}[One-pair guessing estimate; cf. Lemma 13 of
\cite{arjevani2023lowerbounds}]
\label{lb:lem:one-pair-guessing}
Use the projections and events $\mathcal G_j^{(i)}$ from
Lemma~\ref{lb:lem:projection-accumulation}, with the random query sequence
$x_1,\ldots,x_n$.  If $d>n+T$, then, for every $i\in[n]$ and $j\in[T]$,
\begin{equation}\label{lb:eq:one-pair-guessing}
  \PP\left((\mathcal G_j^{(i)})^c,
  \mathcal T_j^{(i-1)}\right)
  \le
  2\exp\!\left(-\frac{d-i-j}{32R^2n}\right)
  \le
  2\exp\!\left(-\frac{d-n-T}{32R^2n}\right).
\end{equation}
\end{lemma}

\begin{proof}
Work on $\mathcal T_j^{(i-1)}$ and condition on
$\mathcal U_{j-1}^{(i-1)}$.  Put
\[
  \mathsf S
  =\operatorname{span}\{u_1,\ldots,u_{j-1},x_1,\ldots,x_{i-1}\},
  \qquad
  P=P_{\mathsf S^\perp}=P_{j-1}^{(i-1)},
\]
and let $r=r_{i,j}$ be the conditional radius from
Lemma~\ref{lb:lem:conditional-haar}.  If $r=0$, then $Pu_j=0$ and hence
\[
  \Pi_{j-1}^{(i)}P_{j-1}^{(i-1)}u_j=0;
\]
the conditional probability of the bad event is therefore zero.

Suppose now that $r>0$ and define
$\widehat u_j:=Pu_j/r$.  By Lemma~\ref{lb:lem:conditional-haar},
$\widehat u_j$ is uniform on the unit sphere of $\mathsf S^\perp$.  The
current query $x_i$ is measurable under the conditioning, and the restriction
of $\Pi_{j-1}^{(i)}$ to $\mathsf S^\perp$ has rank at most one: it is the
projection onto the residual direction of $x_i$ modulo $\mathsf S$.  Since
$r\le1$,
\begin{align*}
&\PP\left(
\norm{\Pi_{j-1}^{(i)}P_{j-1}^{(i-1)}u_j}
\ge\frac{1}{4R\sqrt n}
\ \middle|\ 
\mathcal U_{j-1}^{(i-1)}
\right)\\
&\qquad\le
\PP\left(
\norm{\Pi_{j-1}^{(i)}\widehat u_j}
\ge\frac{1}{4R\sqrt n}
\ \middle|\ 
\mathcal U_{j-1}^{(i-1)}
\right).
\end{align*}
Let $d':=\dim(\mathsf S^\perp)$.  Since
$\dim\mathsf S\le(i-1)+(j-1)$, we may use the conservative lower bound
$d'\ge d-i-j$.  For a uniform vector $v$ on $\mathbb S^{d'-1}$, the standard
spherical-coordinate estimate
\[
  \PP(v_1^2\ge\alpha)\le2e^{-\alpha d'/2}
\]
gives the first inequality in \eqref{lb:eq:one-pair-guessing} with
$\alpha=(16R^2n)^{-1}$.  Averaging over the conditioning, including the
$r=0$ fibers, proves the claim.
\end{proof}

\begin{lemma}[Finite-channel random rotation]
\label{lb:lem:finite-random-rotation}
Let $U$ have the uniform invariant Stiefel distribution on $\Ortho(d,T)$,
independently of i.i.d. seeds
$\xi_1,\xi_2,\ldots$ and the private randomness of an adaptive randomized
algorithm.  Suppose the algorithm makes queries $(x_t,y_t)$ satisfying
$\norm{x_t}\le R$ almost surely and observes the canonical responses
\eqref{lb:eq:canonical-rotated-oracle}.  Put
\[
  \pi_N:=\max_{1\le t\le N}\prog_{1/4}(U^\top x_t).
\]
If the latent kernel is a uniform finite-channel probability-$\varpi$ zero-chain
and $d>N+T$, then
\begin{equation}\label{lb:eq:finite-rotation-bound}
  \PP(\pi_N\ge T)
  \le e^{2\varpi N-T}
  +2NT\exp\!\left(-\frac{d-N-T}{32R^2N}\right).
\end{equation}
The bound does not depend on the finite number $m$ of vector channels.
\end{lemma}

\begin{proof}
We separate rare progress from geometric guessing.

\paragraph{Rare progress.}
For $t\in[N]$, let
\[
  \mathcal V_t
  :=\{\prog_{1/4}(v_t)\le\gamma_{t-1}\}.
\]
Use the filtration
\begin{equation}\label{lb:eq:progress-filtration}
  \mathcal H_t:=\sigma(U,\xi_0,\xi_1,\ldots,\xi_t),
  \qquad
  \mathcal H_0:=\sigma(U,\xi_0).
\end{equation}
Because the algorithm is deterministic after fixing $\xi_0$, $U$, and the
past seeds, $(x_t,y_t)$, $v_t$, $\gamma_{t-1}$, and $\mathcal V_t$ are
$\mathcal H_{t-1}$-measurable.  On $\mathcal V_t$,
\eqref{lb:eq:finite-chain-one-step} gives
$\gamma_t-\gamma_{t-1}\in\{0,1\}$.  If additionally $\xi_t\notin E$, then
\eqref{lb:eq:finite-chain-quiet} gives $\gamma_t=\gamma_{t-1}$.  Equivalently,
\begin{equation}\label{lb:eq:progress-gate}
  0\le\gamma_t-\gamma_{t-1}\le\ind_{\{\xi_t\in E\}}
  \qquad\text{on }\mathcal V_t.
\end{equation}
For the Bernoulli constructions, the right side is $\zeta_t$.
Thus an open gate is necessary for progress on $\mathcal V_t$, but
need not be sufficient.  Freshness of $\xi_t$ implies
\begin{equation}\label{lb:eq:progress-one-step-rigorous}
  \PP\left(
  \gamma_t-\gamma_{t-1}=1,\mathcal V_t
  \mid\mathcal H_{t-1}
  \right)\le \varpi.
\end{equation}
Set $\iota_t:=\gamma_t-\gamma_{t-1}$.  Iterated conditioning gives
\begin{align*}
  \E\left[
  e^{\sum_{t=1}^N\iota_t}
  \ind_{\cap_{t=1}^N\mathcal V_t}
  \right]
  &\le(1-\varpi+\varpi e)^N
  \le e^{2\varpi N}.
\end{align*}
Hence
\begin{equation}\label{lb:eq:progress-chernoff}
  \PP\left(
  \gamma_N\ge T,\bigcap_{t=1}^N\mathcal V_t
  \right)
  \le e^{2\varpi N-T}.
\end{equation}
On $\cap_t\mathcal V_t$, the event $\{\pi_N\ge T\}$ implies
$\{\gamma_N\ge T\}$.

\paragraph{Geometric guessing.}
Apply Lemmas~\ref{lb:lem:projection-accumulation} and
\ref{lb:lem:one-pair-guessing} with horizon $n=N$.  Since $\gamma_t$ is
nondecreasing,
\begin{equation}\label{lb:eq:G-implies-V}
  \bigcap_{i=1}^N\bigcap_{j=1}^T
  \left(\mathcal G_j^{(i)}\cup
  (\mathcal T_j^{(i-1)})^c\right)
  \subseteq
  \bigcap_{i=1}^N\mathcal V_i.
\end{equation}
Indeed, fix $i$ and $j>\gamma_{i-1}$.  Then
$j>\gamma_{q-1}$ for every $q\le i$, so the event on the left of
\eqref{lb:eq:G-implies-V} supplies
$\cap_{q=1}^i\mathcal G_j^{(q)}$.  Lemma~\ref{lb:lem:projection-accumulation}
then gives $|\ip{u_j}{x_i}|\le1/4$.  This holds for every
$j>\gamma_{i-1}$, which is exactly $\mathcal V_i$.

A union bound and Lemma~\ref{lb:lem:one-pair-guessing} now give
\begin{equation}\label{lb:eq:no-guessing-geometric}
  \PP\left(\left[\bigcap_{i=1}^N\mathcal V_i\right]^c\right)
  \le
  2NT\exp\!\left(-\frac{d-N-T}{32R^2N}\right).
\end{equation}
Combining \eqref{lb:eq:progress-chernoff} and
\eqref{lb:eq:no-guessing-geometric} proves \eqref{lb:eq:finite-rotation-bound}.
The argument never counts the number of vector channels: on
$\mathcal T_j^{(i-1)}$, every one of them lies in
$\operatorname{span}\{u_1,\ldots,u_{j-1}\}$, and all non-ambient channels
are already determined by the conditioned latent queries and seeds.
\end{proof}

\begin{corollary}[Terminal output]
\label{lb:cor:terminal-output}
Under the assumptions of Lemma~\ref{lb:lem:finite-random-rotation},
suppose $d>N+1+T$ and, after $N$ calls, the algorithm outputs an
$\sigma(\xi_0,\mathcal T_N)$-measurable point $x_{\rm out}$ with
$\norm{x_{\rm out}}\le R$.  Then
\begin{align}
&\PP\left(
\max\left\{
\max_{t\le N}\prog_{1/4}(U^\top x_t),
\prog_{1/4}(U^\top x_{\rm out})
\right\}\ge T\right)\nonumber\\
&\qquad\le
 e^{2\varpi N-T}
 +2(N+1)T
 \exp\!\left(-\frac{d-(N+1)-T}{32R^2(N+1)}\right).
\label{lb:eq:output-rotation-bound}
\end{align}
\end{corollary}

\begin{proof}
Set $x_{N+1}:=x_{\rm out}$ and regard it only as an additional point in the
geometric argument.  No oracle call is made at this point: no new seed is
drawn, no response is returned, and the revealed progress is fixed at
$\gamma_{N+1}:=\gamma_N$.  Thus the rare-progress estimate remains exactly
$e^{2\varpi N-T}$.  Repeat the projection-accumulation, conditional-residual-law,
and union
bound arguments with the $N+1$ vectors
$x_1,\ldots,x_N,x_{\rm out}$.  For each $j$, on $\mathcal T_j^{(N)}$,
the quiet-simulation argument in Lemma~\ref{lb:lem:conditional-haar}
reconstructs the transcript from $\mathcal U_{j-1}^{(N)}$.
The output is therefore measurable from this field on that event.
The same conditional-law argument applies to the final point without a
new seed.  Taking $n=N+1$ in the geometric term gives
\eqref{lb:eq:output-rotation-bound}.
\end{proof}

\begin{corollary}[Soft projection and known postprocessing]
\label{lb:cor:postprocessing}
Let $\rho:\R^d\to\R^d$ be Borel with $\norm{\rho(x)}\le R$ for all $x$.
Let $(x,y,o)\mapsto\mathsf P_{x,y}(o)$ be a jointly Borel map independent
of $U$, and suppose the actual response is
\[
  \mathsf P_{x,y}\!\left(
  \widetilde{\cO}_U(\rho(x),y;\xi)
  \right).
\]
Then
Lemma~\ref{lb:lem:finite-random-rotation} and
Corollary~\ref{lb:cor:terminal-output} hold with $U^\top x$ replaced by
$U^\top\rho(x)$.
\end{corollary}

\begin{proof}
An algorithm interacting with the actual oracle can be simulated by an
algorithm interacting with the richer canonical oracle: the simulator
queries $\rho(x)$, receives the canonical response, applies the known
postprocessing, and passes the result to the original algorithm.  At
termination it returns $\rho(x_{\rm out})$, with no additional call.
Joint measurability makes this a valid adaptive simulation.
\end{proof}

\subsection{The comparison oracle and the first-order benchmark}
\label{lb:sec:comparison}

\begin{proof}[Proof of Proposition~\ref{prop:lb-comparison}]
Take the scaled pair $(A_U,u_U)$ from
Proposition~\ref{lb:prop:scaled-pair}, with $T$ as in
\eqref{lb:eq:T-choice}.  For $0<\varpi\le1/2$, put
$N_{\mathrm s}:=\lfloor T/(8\varpi)\rfloor$ and choose
\begin{equation}
\label{lb:eq:comparison-dimension}
  d:=\left\lceil N_{\mathrm s}+1+T+
  32R^2(N_{\mathrm s}+1)\log\bigl(64(N_{\mathrm s}+1)T\bigr)\right\rceil.
\end{equation}
Draw $U$ uniformly from $\Ortho(d,T)$, independently of the algorithm
seed and oracle draws.  The functions are smooth; their Lipschitz,
value, and gap bounds follow from
\eqref{lb:eq:scaled-u-bounds}--\eqref{lb:eq:scaled-gap}.

At a bounded query $\bar r\in\R^d$, write $v=U^\top\bar r$ and consider
\begin{equation}
\label{lb:eq:comparison-canonical}
  (w,v_B,v_W)
  :=\left(\frac\zeta\varpi W_T(v),\,
           U\nabla B_T(v),\,
           U\frac\zeta\varpi\nabla W_T(v)\right).
\end{equation}
By Corollary~\ref{lb:cor:two-channel-zero-chain}, its two vector
channels satisfy Definition~\ref{lb:def:finite-chain} with exceptional
event $\{\zeta=1\}$.  The scalar channel is a Borel function of the
latent query $v$ and the seed, as allowed by that definition.
Hence the full response, including the scalar value, is covered by
Corollary~\ref{lb:cor:terminal-output}.

One comparison call at an arbitrary $x$ is obtained by querying
\eqref{lb:eq:comparison-canonical} at
$\bar r=\rho_R(x/\lambda_{\rm lb})$ and returning
\begin{equation}
\label{lb:eq:comparison-postprocessing}
\begin{aligned}
  r_A&:=\frac a{\lambda_{\rm lb}}
    J_R(x/\lambda_{\rm lb})^\top v_B
    +\frac{a\eta_{\rm r}}{\lambda_{\rm lb}^2}x,\\
  r_u&:=aw,\qquad
  r_{\nabla u}:=\frac a{\lambda_{\rm lb}}
    J_R(x/\lambda_{\rm lb})^\top v_W.
\end{aligned}
\end{equation}
Differentiating \eqref{lb:eq:scaled-Phi}--\eqref{lb:eq:scaled-u} shows
that $(r_A,r_u,r_{\nabla u})$ is exactly
$\mathcal R_\varpi(x)$ in \eqref{eq:lb-comparison-oracle}.
These transformations are jointly continuous and independent of $U$.
The compressed final output
$\bar r_{\rm out}=\rho_R(x_{\rm out}/\lambda_{\rm lb})$ has norm at most
$R$ and costs no oracle call.

For any algorithm using at most $N_{\mathrm s}$ comparison calls,
pad a shorter run with ignored calls.
Corollaries~\ref{lb:cor:postprocessing} and \ref{lb:cor:terminal-output}
and the dimension choice give
\begin{equation}
\label{lb:eq:comparison-unfinished}
  \PP\bigl(\prog_{1/4}(U^\top\bar r_{\rm out})\ge T\bigr)
  \le e^{2\varpi N_{\mathrm s}-T}+\frac1{32}
  \le e^{-3T/4}+\frac1{32}<\frac14,
\end{equation}
since $T\ge4$.  Equation~\eqref{lb:eq:scaled-stationarity} now gives
\eqref{eq:lb-common-progress}.  Thus $c=1/8$ is admissible, and a
common numerical $C$ bounds the remaining quantities.
\end{proof}

\subsubsection{Variance and the ordinary first-order lower bound}
\label{lb:sec:first-order}

\begin{lemma}[Variance and the ordinary first-order lower bound]
\label{lem:lb-gradient-variance}
For each pair in Proposition~\ref{prop:lb-comparison},
\eqref{eq:lb-first-order-oracle} is conditionally unbiased and
\begin{equation}
\label{eq:lb-signal-variance}
  \E[\|\widehat{\nabla P}(x)-\nabla P(x)\|^2\mid\mathcal H]
  =\frac{1-\varpi}{\varpi}\|\nabla u(x)\|^2
  \le\frac{C^2\eps^2}{\varpi}.
\end{equation}
Choosing $\varpi=2(1+C^2)\eps^2\le1/2$ gives a family with variance at most one for which guaranteeing
$\E\|\nabla P(x_{\rm out})\|^2\le\eps^2$ requires
$\Omega(\eps^{-4})$ queries.
\end{lemma}

\begin{proof} %
Fix $U$ and condition on the history $\mathcal H$, which fixes the query
$x$.  The fresh Bernoulli seed satisfies $\E[\zeta/\varpi]=1$, so
\eqref{eq:lb-first-order-oracle} is conditionally unbiased.  Moreover,
\[
  \E[\|\widehat{\nabla P_U}(x)-\nabla P_U(x)\|^2\mid\mathcal H]
  =\frac{1-\varpi}{\varpi}\|\nabla u_U(x)\|^2
  \le\frac{C^2\eps^2}{\varpi},
\]
where $C$ is the constant in Proposition~\ref{prop:lb-comparison}.
If $\mathcal R_\varpi(x)=(r_A,r_u,r_{\nabla u})$, then
$r_A-r_{\nabla u}$ has exactly this stochastic-gradient law, with one
comparison call per gradient call.

Set $\varpi=2(1+C^2)\eps^2$.  For sufficiently small $\eps$ this lies
in $(0,1/2]$, and the conditional variance is at most $1/2$.
The objective has gradient Lipschitz constant at most $2C$ and initial
gap at most $1/4$, by Proposition~\ref{prop:lb-comparison}.
Simulating any stochastic-gradient algorithm with at most
$T/(8\varpi)$ calls therefore gives a comparison-oracle algorithm with
the same budget and output law.  Proposition~\ref{prop:lb-comparison}
yields
\[
  \PP\bigl(\|\nabla P_U(x_{\rm out})\|\ge2\eps\bigr)\ge\frac34.
\]
Averaging this probability over $U$ gives a fixed rotation on which the
probability, now over algorithm and oracle randomness, is at least
$3/4$.  On that instance,
$\E\|\nabla P_U(x_{\rm out})\|^2\ge3\eps^2>\eps^2$.
Finally, $T/\varpi=\Theta(\eps^{-4})$.
\end{proof}

\subsection{The bilevel construction}
\label{lb:sec:lift}

Fix $0<\eps\le\delta\le1$ and a scaled pair $(A_U,u_U)$ from
Proposition~\ref{lb:prop:scaled-pair}.  With $x\in\R^d$ and $y\in\R$, set
\begin{equation}\label{lb:eq:lift-f}
  f_U(x,y):=A_U(x)+y,
\end{equation}
\begin{equation}\label{lb:eq:lift-g}
  g_U(x,y):=\frac12y^2+\delta\sin(y/\delta)u_U(x).
\end{equation}
The dependence on $\delta$ is suppressed in this appendix.

\subsubsection{Primitive oracle: simulation and variance}

For $0<\varpi\le1/2$, draw a fresh
$\zeta\sim\mathrm{Bernoulli}(\varpi)$ at each call and return
\begin{align}
  \hat F_x&:=\nabla A_U(x),
  &\hat F_y&:=1,
  \label{lb:eq:oracle-f}\\
  \hat G_x&:=\delta\sin(y/\delta)\frac\zeta\varpi\nabla u_U(x),
  &\hat G_y&:=y+\cos(y/\delta)\frac\zeta\varpi u_U(x).
  \label{lb:eq:oracle-g}
\end{align}
The seeds are independent across calls and independent of $U$ and the
algorithm seed.  Define the numerical constant
\begin{equation}\label{lb:eq:Cvar}
  C_{\rm var}:=C_{u,1}^2+C_{u,0}^2.
\end{equation}

\begin{lemma}[Primitive simulation and conditional variance]
\label{lb:lem:primitive-simulation}
For every fixed $U$ and selected query $(x,y)$, the four estimates are
conditionally unbiased and their joint conditional variance is at most
$C_{\rm var}\delta^2\eps^2/\varpi$.
One primitive call is a known measurable transformation of one
comparison call $\mathcal R_\varpi(x)$.
\end{lemma}

\begin{proof}
Condition on the history, fixing $(x,y)$.  Unbiasedness follows from
$\E[\zeta/\varpi]=1$ and direct differentiation of
\eqref{lb:eq:lift-f}--\eqref{lb:eq:lift-g}.
The upper estimates are exact.  The sum of the conditional squared
errors in the lower estimates is
\begin{align}
 &\frac{1-\varpi}{\varpi}
 \left(\delta^2\sin^2(y/\delta)\|\nabla u_U(x)\|^2
       +\cos^2(y/\delta)|u_U(x)|^2\right)\nonumber\\
 &\qquad\le
 \frac{C_{u,1}^2\delta^2\eps^2+C_{u,0}^2\eps^4}{\varpi}
 \le\frac{C_{\rm var}\delta^2\eps^2}{\varpi}.
 \label{lb:eq:primitive-variance}
\end{align}
Here Proposition~\ref{lb:prop:scaled-pair} supplies the uniform bounds,
and $\eps\le\delta$ controls the scalar contribution.

If $\mathcal R_\varpi(x)=(r_A,r_u,r_{\nabla u})$, return
\begin{equation}
\label{lb:eq:primitive-from-comparison}
  \left(r_A,\ 1,\ \delta\sin(y/\delta)r_{\nabla u},\
        y+\cos(y/\delta)r_u\right).
\end{equation}
This reproduces the full joint law of
\eqref{lb:eq:oracle-f}--\eqref{lb:eq:oracle-g}, including the common
Bernoulli draw.  The transformation is continuous and independent of
$U$.  Applying it after every comparison call simulates an adaptive
primitive-oracle algorithm, with the same query choices and output law.
\end{proof}

\subsubsection{Lower solution and outer gradient}

For $|s|/\delta\le1/4$, let $Y_\delta(s)$ denote the unique solution of
\begin{equation}\label{lb:eq:Y-equation}
  Y+s\cos(Y/\delta)=0.
\end{equation}

\begin{lemma}[Lower response]
\label{lb:lem:lower-solution}
For $|s|/\delta\le1/4$, the response $Y_\delta(s)$ is continuously
differentiable and obeys
\begin{align}
  |Y_\delta(s)|&\le |s|,
  \label{lb:eq:Y-size}\\
  |Y_\delta(s)+s|&\le\frac{|s|^3}{2\delta^2},
  \label{lb:eq:Y-value-error}\\
  |Y_\delta'(s)+1|&\le\frac{2s^2}{\delta^2}.
  \label{lb:eq:Y-derivative-error}
\end{align}
For all sufficiently small $\eps$, uniformly over $\eps\le\delta\le1$, the lower
solution of \eqref{lb:eq:lift-g} is
\begin{equation}\label{lb:eq:y-star}
  y_U^\star(x)=Y_\delta(u_U(x)),
\end{equation}
and $\nabla^2_{yy}g_U(x,y)\ge3/4$ globally.
\end{lemma}

\begin{proof}
The derivative in $Y$ of the left side of \eqref{lb:eq:Y-equation} is
\[
  1-(s/\delta)\sin(Y/\delta)\ge\frac34.
\]
Hence that function is strictly increasing and crosses zero exactly once;
the implicit-function theorem gives differentiability.  Equation
\eqref{lb:eq:Y-equation} gives $|Y_\delta(s)|\le|s|$, and
\[
  |Y_\delta(s)+s|
  =|s|\,|1-\cos(Y_\delta(s)/\delta)|
  \le\frac{|s|^3}{2\delta^2}.
\]
Implicit differentiation yields
\[
  Y_\delta'(s)
  =-\frac{\cos(Y_\delta(s)/\delta)}
  {1-(s/\delta)\sin(Y_\delta(s)/\delta)}.
\]
With $t=Y_\delta(s)/\delta$ and $\beta=s/\delta$,
\[
  |Y_\delta'(s)+1|
  =\left|\frac{1-\cos t-\beta\sin t}{1-\beta\sin t}\right|
  \le\frac{\tfrac12\beta^2+\beta^2}{3/4}
  =2\beta^2,
\]
where we used $|t|\le|\beta|$, $|1-\cos t|\le t^2/2$, and
$|\sin t|\le|t|$.  This proves \eqref{lb:eq:Y-derivative-error}.  Finally,
Proposition~\ref{lb:prop:scaled-pair} and $\delta\ge\eps$ imply
\[
  \frac{\norm{u_U}_\infty}{\delta}
  \le C_{u,0}\eps\le\frac14
\]
for a universal choice of $\eps_0$.  Since
\[
  \nabla_y g_U=y+\cos(y/\delta)u_U,
  \qquad
  \nabla^2_{yy}g_U=1-\delta^{-1}\sin(y/\delta)u_U,
\]
the claims follow.
\end{proof}

Define the actual hyperobjective
\begin{equation}\label{lb:eq:actual-hyperobjective}
  \Phi_U(x)
  :=f_U(x,y_U^\star(x))
  ={A_U}(x)+Y_\delta(u_U(x)).
\end{equation}

\begin{corollary}[Controlled perturbation of the hard objective]
\label{lb:cor:lift-perturbation}
There are universal constants $C_F,C_\nabla>0$ such that
\begin{align}
  \sup_x|\Phi_U(x)-P_U(x)|
  &\le C_F\frac{\eps^6}{\delta^2}
  \le C_F\eps^4,
  \label{lb:eq:lift-value-perturb}\\
  \sup_x\norm{\nabla \Phi_U(x)
  -\nabla P_U(x)}
  &\le C_\nabla\frac{\eps^5}{\delta^2}
  \le C_\nabla\eps^3.
  \label{lb:eq:lift-gradient-perturb}
\end{align}
After decreasing $\eps_0$ universally if necessary,
\begin{equation}\label{lb:eq:lift-gap}
  \Phi_U(0)-\inf_x\Phi_U(x)\le1,\qquad |y_U^\star(0)|\le1,
\end{equation}
and
\begin{equation}\label{lb:eq:lift-stationarity}
  \|\nabla P_U(x)\|\ge2\eps
  \quad\Longrightarrow\quad
  \norm{\nabla \Phi_U(x)}\ge\frac32\eps.
\end{equation}
\end{corollary}

\begin{proof}
The first estimate follows from \eqref{lb:eq:Y-value-error},
$P_U={A_U}-u_U$, and
\eqref{lb:eq:scaled-u-bounds}.  The chain rule and
\eqref{lb:eq:Y-derivative-error} give
\[
  \nabla \Phi_U
  =\nabla{A_U}
   +Y_\delta'(u_U)\nabla u_U,
\]
and hence the second estimate.  Here $\delta\ge\eps$ makes both displayed
upper bounds uniform over $\eps\le\delta\le1$.  Combining the value estimate with
\eqref{lb:eq:scaled-gap} gives the objective-gap bound.  The initial lower
solution satisfies $|y_U^\star(0)|\le C_{u,0}\eps^2\le1$ for small enough
$\eps$, proving \eqref{lb:eq:lift-gap}.  If $\|\nabla P_U(x)\|\ge2\eps$,
the gradient estimate gives
$\|\nabla\Phi_U(x)\|\ge2\eps-C_\nabla\eps^3\ge3\eps/2$
for a numerical sufficiently small $\eps_0$.
\end{proof}

The gradient estimate can also be read directly from the implicit
hypergradient.  In the notation of Section~\ref{sec:problem},
\[
  b=1,\qquad
  H_U(x,y)=1-\delta^{-1}\sin(y/\delta)u_U(x),\qquad
  J_U(x,y)=\cos(y/\delta)\nabla u_U(x).
\]
At $y=y_U^\star(x)$, implicit differentiation gives
\begin{equation}
\label{lb:eq:retained-correction}
\begin{aligned}
  J_U(x,y_U^\star(x))z_U^\star(x)
  &=-Y_\delta'(u_U(x))\nabla u_U(x),\\
  \|J_U(x,y_U^\star(x))z_U^\star(x)-\nabla u_U(x)\|
  &\le2C_{u,0}^2C_{u,1}\frac{\eps^5}{\delta^2}.
\end{aligned}
\end{equation}
Thus the mixed-derivative term in the hypergradient retains the
correction $\nabla u_U$ even though the primitive signal is smaller.

\subsubsection{Base and higher-order regularity}

\begin{proof}[Proof of Proposition~\ref{prop:lb-scalar-transfer}]
The Hessian blocks of $g_U$ are
\[
  \nabla^2_{xx}g_U=\delta\sin(y/\delta)\nabla^2u_U(x),\qquad
  \nabla^2_{xy}g_U=\cos(y/\delta)\nabla u_U(x),
\]
\[
  \nabla^2_{yy}g_U=1-\delta^{-1}\sin(y/\delta)u_U(x).
\]
They are bounded by numerical constants for all $(x,y)$, by
Proposition~\ref{lb:prop:scaled-pair} and $\eps\le\delta\le1$.
Hence the full gradient of $g_U$ is jointly Lipschitz.
The $x$- and $y$-Lipschitz moduli of $J_U$ are bounded by
$C_{u,2}$ and $C_{u,1}\eps/\delta$, respectively; those of $H_U$
are bounded by $C_{u,1}\eps/\delta$ and $C_{u,0}\eps^2/\delta^2$.
Combining the two moduli in each case gives a numerical joint
Lipschitz bound, uniformly over $\eps\le\delta\le1$.
Lemma~\ref{lb:lem:lower-solution} gives strong convexity with $\mu=3/4$.
Also $\nabla_y f_U=1$, and \eqref{lb:eq:Phi-hessian-bound} bounds the
full Hessian of $f_U$ by one.  The objective is bounded below and obeys
\eqref{eq:lb-class-initial} by Corollary~\ref{lb:cor:lift-perturbation}.
This proves Assumption~\ref{ass:base}.

The same corollary proves \eqref{eq:lb-retained-correction}, while
Lemma~\ref{lb:lem:primitive-simulation} supplies conditional
unbiasedness, variance, and one-call simulation.  A numerical constant
$C_{\rm lift}\ge\max\{C_\nabla,C_{\rm var}\}$ bounds both quantities
in the proposition.
\end{proof}

For an integer $p\ge1$, choose
\begin{equation}\label{lb:eq:delta-choice}
  \delta:=\eps^{1/p}\in[\eps,1].
\end{equation}

\begin{lemma}[Uniform order-$p$ regularity]
\label{lb:lem:analytic-class}
With \eqref{lb:eq:delta-choice}, the family satisfies
Assumptions~\ref{ass:base} and \ref{ass:order-p} with a numerical bound
on every constant, independent of $T,d,U,\eps$, and $p$.
\end{lemma}

\begin{proof}
Proposition~\ref{prop:lb-scalar-transfer} proves the base regularity.
For $p\ge2$, write $t=y/\delta$ and denote derivatives of sine by
$\sin^{(j)}$.  Direct differentiation gives
\begin{equation}\label{lb:eq:high-g-formula}
\begin{aligned}
  \nabla_y^p\nabla_x g_U(x,y)
  &=\delta^{1-p}\sin^{(p)}(t)\nabla u_U(x),\\
  \nabla_y^p\nabla_y g_U(x,y)
  &=\delta^{-p}\sin^{(p+1)}(t)u_U(x).
\end{aligned}
\end{equation}
Every derivative of sine has absolute value at most one.  Since the
lower variable is scalar, the operator norms in
Assumption~\ref{ass:order-p} are respectively the Euclidean norm and
absolute value.  The Lipschitz moduli in $y$ are therefore at most
\[
  C_{u,1}\eps\delta^{-p}=C_{u,1},\qquad
  C_{u,0}\eps^2\delta^{-(p+1)}
  =C_{u,0}\eps^{1-1/p}\le C_{u,0}.
\]
For $p=1$, the base Lipschitz bounds for $J_U$ and $H_U$ imply the
same assumption.  All constants are independent of $p$.
\end{proof}

Finally, set
\begin{equation}\label{lb:eq:crr}
  c_{\rm rr}:=2(1+C_{\rm var}),
\end{equation}
and take
\begin{equation}\label{lb:eq:rr-choice}
  \varpi:=c_{\rm rr}\delta^2\eps^2
  =c_{\rm rr}\eps^{2+2/p}\le\frac12.
\end{equation}
The last inequality holds for a numerical sufficiently small
$\eps_0$, uniformly in $p$, since $\delta\le1$.

\begin{lemma}[Admissible primitive oracle]
\label{lb:lem:oracle-valid}
With \eqref{lb:eq:rr-choice}, the oracle
\eqref{lb:eq:oracle-f}--\eqref{lb:eq:oracle-g} satisfies
Assumption~\ref{ass:oracle} with $\sigma^2\le1/2$.
\end{lemma}

\begin{proof}
Lemma~\ref{lb:lem:primitive-simulation} gives conditional unbiasedness
and joint variance at most
$C_{\rm var}\delta^2\eps^2/\varpi=C_{\rm var}/c_{\rm rr}\le1/2$.
The response is a Borel function of the query and fresh Bernoulli seed;
its pushforward law is a Borel kernel.  Independent seeds give the
required conditional independence between fresh calls.
\end{proof}

\subsection{Completion of the lower bound}
\label{lb:sec:main-proof}

\begin{proof}[Proof of Theorem~\ref{thm:lower}]
Fix an integer $p\ge1$ and $0<\eps<\eps_0$.  Take $T$ from
\eqref{lb:eq:T-choice}, $\delta=\eps^{1/p}$, and
$\varpi=c_{\rm rr}\delta^2\eps^2$ as in \eqref{lb:eq:rr-choice}.
Use the distribution of $(A_U,u_U)$ in
Proposition~\ref{prop:lb-comparison}, and form the bilevel pair and
primitive oracle in Section~\ref{lb:sec:lift}.
Define
\begin{equation}\label{lb:eq:N0}
  N_0:=\left\lfloor\frac{T}{8\varpi}\right\rfloor,
\end{equation}
and the numerical constant
\begin{equation}\label{lb:eq:clb-choice}
  c_{\rm lb}:=\frac{c_T}{64c_{\rm rr}}.
\end{equation}
For a numerical sufficiently small $\eps_0$, we have
$T\ge(c_T/2)\eps^{-2}$ and $T/(8\varpi)\ge2$, uniformly in $p$.
Consequently,
\begin{equation}\label{lb:eq:N0-dominates-clb}
  N_0\ge\frac{T}{16\varpi}
  \ge\frac{c_T}{32c_{\rm rr}}\eps^{-4-2/p}
  \ge c_{\rm lb}\eps^{-4-2/p}.
\end{equation}

Let an arbitrary adaptive randomized algorithm use at most $N_0$
primitive calls.  Lemma~\ref{lb:lem:primitive-simulation} simulates it
with the same number of calls to $\mathcal R_\varpi$ and preserves its
output law.  Proposition~\ref{prop:lb-comparison}, whose proof permits
$c=1/8$, therefore gives
\begin{equation}
\label{lb:eq:transferred-gradient-event}
  \PP\bigl(\|\nabla P_U(x_{\rm out})\|\ge2\eps\bigr)\ge\frac34.
\end{equation}
By \eqref{lb:eq:lift-stationarity}, on this event,
$\|\nabla\Phi_U(x_{\rm out})\|\ge3\eps/2$.
Averaging the event probability over $U$ gives a fixed rotation for
which this event has probability at least $3/4$ over the algorithm and
oracle randomness.  For that fixed instance,
\[
  \E\|\nabla\Phi_U(x_{\rm out})\|^2
  \ge\frac34\left(\frac32\eps\right)^2
  =\frac{27}{16}\eps^2>\eps^2.
\]
Equation~\eqref{lb:eq:N0-dominates-clb} covers every integer budget in
Theorem~\ref{thm:lower} with $c=c_{\rm lb}$.

It remains to check that the fixed instance lies in the stated class.
Proposition~\ref{prop:lb-scalar-transfer} gives the numerical base
regularity and initial bounds.  Lemma~\ref{lb:lem:analytic-class}
bounds the order-$p$ constant independently of $p$, and
Lemma~\ref{lb:lem:oracle-valid} gives $\sigma^2\le1/2$.
Taking $M_0$ to be the maximum of one and these numerical bounds places
the pair in $\mathcal P_p(M_0)$.  Every smallness condition on $\eps$
used above is numerical and independent of $p$, so a common
$\eps_0>0$ suffices.

The construction is finite dimensional.  In
\eqref{lb:eq:comparison-dimension}, $N_{\mathrm s}=N_0$ and
$R^2=\Theta(T)$, hence
\[
  d=O\!\left(\frac{T^2}{\varpi}\log\frac{T}{\varpi}\right)
   =O\!\left(\eps^{-6-2/p}\log(1/\eps)\right).
\]
\end{proof}

\end{document}